\documentclass[12pt,twoside,a4paper]{article}
\usepackage{amssymb,graphicx,amsmath,amsthm, bm,mathtools,longtable,array,booktabs}

\usepackage[colorlinks = true,
            linkcolor = blue,
            urlcolor  = blue,
            citecolor = blue,
            anchorcolor = blue]{hyperref}
 
\usepackage{enumitem}  
\usepackage{wrapfig}
\usepackage{cite} 
\usepackage[T1]{fontenc}

\numberwithin{equation}{section}

\makeatletter
\def\twodigits#1{\expandafter\@twodigits\csname c@#1\endcsname}
\def\@twodigits#1{%
  \ifnum#1<10 0\fi
  \number#1}
\makeatother
\AddEnumerateCounter{\twodigits}{\@twodigits}{10}

\usepackage{amsmath}

\usepackage[a4paper]{geometry}
\newcommand{\KSigmaD}{\mathrm{K}_{D \Sigma } }
\newcommand{\KGammaD}{\mathrm{K}_{D \Gamma } }
\newcommand{\KSigmaOmegaTwo}{\mathrm{K}_{\Omega_{2} \Sigma} }
\newcommand{\KSigmaGamma}{\mathrm{K}_{\Gamma\Sigma} }
\newcommand{\KGammaSigma}{\mathrm{K}_{\Sigma\Gamma} }
\newcommand{\KGammaOmegaOneC}{\mathrm{K}_{\Omega_1^{\rm c}\Gamma} }
\newcommand{\KSigmaOmegaOneTwo}{\mathrm{K}_{\Omega_{12}\Sigma}}
\newcommand{\KGammaOmegaOneTwo}{\mathrm{K}_{\Omega_{12}\Gamma}}

\newcommand{\revAB}[1]{{#1}}
\newcommand{\revB}[1]{{#1}}
\newcommand{\revA}[1]{{#1}}
 
\theoremstyle{remark}
\newtheorem{remark}{Remark}[section]
\theoremstyle{plain}
\newtheorem{lemma}[remark]{Lemma}
\newtheorem{theorem}[remark]{Theorem}

\title{ An overlapping decomposition framework for  wave propagation in heterogeneous and unbounded media:
Formulation, analysis, algorithm,  and simulation}
\author{V. Dom\'{\i}nguez \thanks{Department of  Estad\'{\i}stica, Inform\'atica y Matem\'aticas, Universidad P\'{u}blica de Navarra, Tudela, Spain/ Institute for Advanced Materials (INAMAT), Pamplona, Spain. Email: victor.dominguez@unavarra.es}
\and M. Ganesh \thanks{Department Applied Mathematics and Statistics Department, Colorado School of Mines, Golden, CO, USA. Email: mganesh@mines.edu}
\and F.J. Sayas
\thanks {Department of Mathematical Sciences, University of Delaware, Newark, DE, USA.
 ~~~~~~~~~~~~~~~~~~~~~~~~~~~~~~~~~~~~~~~~~~~~~ Email: fjsayas@udel.edu}
}
\date{\today}

\DeclareMathOperator{\spann}{span}

\begin{document}
\maketitle 
\begin{abstract}
{\revAB{A natural medium for  wave propagation comprises a  coupled bounded heterogeneous region  and an  unbounded homogeneous free-space. Frequency-domain wave propagation models in the medium, such as the variable coefficient Helmholtz equation,  include a faraway decay  radiation condition (RC).   It is desirable to develop algorithms that  incorporate  the full physics of the heterogeneous
and unbounded medium wave propagation model, and avoid an approximation of  the  RC. 
In this work we first present and analyze an overlapping decomposition framework that is equivalent to the full-space  heterogeneous-homogenous  continuous  model,  governed by the Helmholtz equation with a spatially dependent refractive index  and the RC. Our novel overlapping framework allows the user to choose two free boundaries, and gain the  advantage of applying established  high-order finite and boundary element methods  (FEM and BEM) to simulate an equivalent coupled model.}}

{\revAB{The coupled model comprises  auxiliary interior bounded heterogeneous and exterior unbounded  homogeneous Helmholtz problems. A smooth boundary can be chosen for simulating  the exterior  problem using a spectrally accurate BEM, and a  simple  boundary can be used to setup a high-order FEM for the interior problem. Thanks to the spectral accuracy of the exterior  computational model, the
resulting coupled system in the overlapping region is relatively very small. 
Using the decomposed equivalent framework, we develop  a novel overlapping FEM-BEM algorithm for simulating   the acoustic or electromagnetic wave propagation in two dimensions. Our  FEM-BEM algorithm for the full-space model incorporates the RC exactly. Numerical experiments demonstrate the efficiency of the FEM-BEM approach for simulating smooth and non-smooth wave fields, with the latter induced by  a complex heterogeneous medium and a discontinuous refractive index.}}

\end{abstract}


\newpage
~\label{sec:intro}
\section{Introduction}
{\revB{Wave propagation simulations, governed by the Helmholtz equation, in bounded heterogeneous and unbounded homogenous media are fundamental for numerous applications~\cite{KressColton,Ihlenburg:1998,nedlec:book}.}}

{\revB{Finite element methods (FEM) are efficient for  simulating the Helmholtz equation in a bounded heterogeneous  medium, say, $\Omega_0\subset \mathbb{R}^m$ ($m=2,3$). The  standard (non-coercive) variational  formulation of the variable coefficient Helmholtz equation in $H^1(\Omega_0)$~\cite{Ihlenburg:1998} has been widely used for developing and analyzing the sign-indefinite FEM,
see for example~\cite{BGP, het2, het3, mg2018, het7,  het5}. The open problem of developing a coercive variational formulation for the heterogeneous Helmholtz model was solved recently in~\cite{mg2019}, and an associated preconditioned sign-definite high-order FEM was also  established using direct and domain decomposition methods in~\cite{mg2019}.}

{\revB{For a large class of applications the wave propagation   occurs in the bounded heterogeneous medium and also in its complement, 
$\mathbb{R}^m \setminus \Omega_0$, the  exterior unbounded homogeneous medium. Using the fundamental solution, the constant coefficient Helmholtz equation exterior to $\Omega_0$ can be reformulated as an integral equation (IE) on the boundary of $\Omega_0$.
Algorithms for simulating the boundary IE (BIE) are known as boundary element methods (BEM). Several coercive and non-coercive BIE reformulations~\cite{KressColton,nedlec:book}  of the exterior Helmholtz model have been used to develop algorithms for the exterior homogeneous Helmholtz models, see for example the acoustic  BEM survey articles~\cite{bem-eng-sur, bem-math-sur}, respectively, by mathematical and engineering researchers,  each with over 400 references.}}

The exterior wave propagation BEM models lead to  dense complex algebraic systems,
and the standard variational formulation based interior wave FEM models lead to sparse complex systems with their eigenvalues 
\revB{in}  the
left half of the complex plane~\cite{sdparti, Moiola}. Developing efficient preconditioned iterative solvers for such  
systems  \revB{has} also dominated research activities over the last two decades~\cite{gander_zhang}, in conjunction
with efficient implementations  using multigrid and domain decomposition techniques, 
see~\cite{mg2017, mg2018} and references therein.

For applications that require solving both the interior  heterogenous and exterior homogeneous problems, 
various couplings of the FEM and BEM algorithms with appropriate conditions on {\em polygonal interfaces}  have also been investigated in the literature\cite{BrJo:1979,MR974843,BrJoNe:1978}. The review article~\cite{sayas} describes some theoretical validations of the coupling approaches considered in the earlier literature 
and delicate choices of the coupling interface. The coupling methods in~\cite{BrJo:1979,MR974843,BrJoNe:1978, HanNew, sayas}  lead to very large algebraic systems with both dense and sparse structures.  For wave propagation models, given the complexity involved in even separately solving the FEM and BEM algebraic systems, it is efficient 
to avoid large combined dense and sparse structured systems arising from the coupling methods in~\cite{BrJo:1979,MR974843,BrJoNe:1978, HanNew, sayas}. 

Such complicated-structured coupled large-scale systems can be avoided, for the Helmholtz PDE interior and exterior 
problems,  using the approach proposed in~\cite{Kirsch2} and recently further explored  in~\cite{GaMor:2016} using high-order elements for a class of applications with complex heterogeneous structures. 
The FEM-BEM algorithms in~\cite{Kirsch2, GaMor:2016} are based on the idea of using a  non-overlapping {\em smooth interface} to couple the interior and exterior solutions. As described in~\cite[Section~6]{GaMor:2016},
there are several open mathematical analysis problems remain to be solved in the coupling 
and FEM-BEM framework of~\cite{Kirsch2, GaMor:2016}.

The choice of smooth interface in the
FEM-BEM algorithms of~\cite{Kirsch2, GaMor:2016} is crucial because the methods require solving
several interior and exterior wave problems to setup the interface condition. In particular, the number of FEM and BEM problems
to be solved is  twice \revAB{the number of}  degrees of freedom  required to approximate the unknown interface function.
The interface function can be approximated by a few degrees of freedom only on  smooth interfaces.
Efficient spectrally accurate BEM algorithms have been developed for simulating scattered waves exterior
to smooth  boundaries  in two and three dimensional domains~\cite{Bds2013,377360, KressColton, ganesh:high-order}.
However for  standard interior FEM algorithms, it is desirable to have simple polygonal/polyhedral boundaries, and
in particular those with right angles, which
\revB{facilitate the} development and  implementation of high-order FEM algorithms. 
 
To this end, we develop an equivalent framework for the heterogeneous and unbounded
region wave propagation model with two artificial  interfaces. In particular, our novel FEM-BEM framework is based  on an interior smooth interface 
$\Gamma$ for simulating scattered exterior waves using a spectrally accurate Nystr\"om BEM, and an exterior simple polygonal/polyhedral interface $\Sigma$  for the efficient high-order FEM simulation of the absorbed interior waves. In Figure~\ref{fig:01}, we
sketch the resulting overlapped decomposition of a heterogeneous and  unbounded medium in which the absorbed and scattered waves are 
induced by an input incident wave $u^{\rm inc}$. 

\begin{figure}[!ht]
 \centerline{\includegraphics[width=0.38\textwidth]{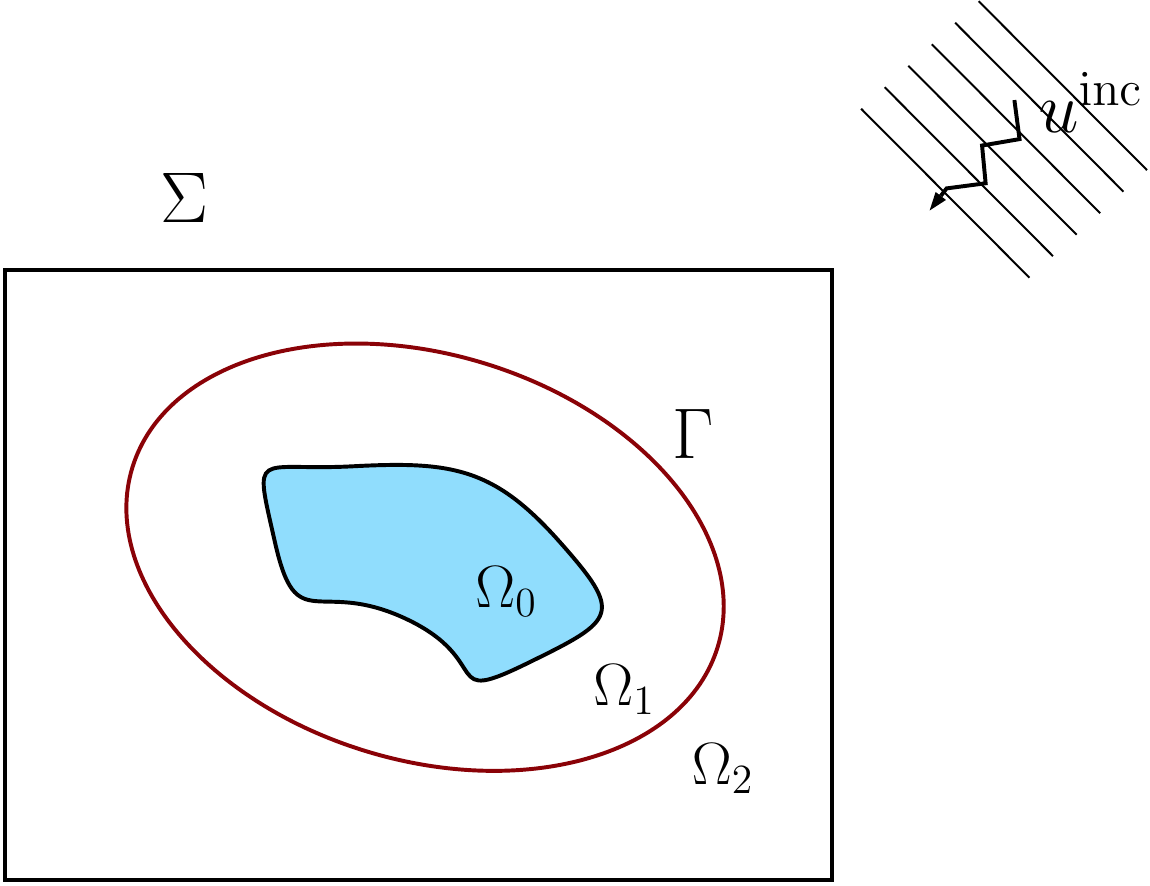}}
 \vspace{-0.1in}
  \caption{\label{fig:01} A model configuration with an input incident wave $u^{\rm inc}$ impinging on a heterogeneous medium $\Omega_0$. 
  The artificial boundaries in our decomposition framework for the auxiliary bounded (FEM) and unbounded (BEM) models are $\Sigma$ and $\Gamma$, respectively. The bounded  domain  for the FEM is $\Omega_2$ (with  boundary $\Sigma$), and  the unbounded   region  for the BEM is $\mathbb{R}^m \setminus \overline{\Omega}_1$
  (exterior to the smooth  interface $\Gamma$). 
  The domain $\Omega_1$ (with  boundary $\Gamma$) is chosen so that   $\overline{\Omega}_0\subset \Omega_1\subset\Omega_2$, and the overlapping region in the framework, to match the FEM and BEM solutions, is  $(\mathbb{R}^m \setminus \overline{\Omega}_1)\cap \Omega_2$.}
\end{figure}

The decomposition facilitates the application of efficient  high-order FEM algorithms
in the interior polygonal/polyhedral domain $\Omega_2$, that contains
the heterogeneous region $\Omega_0 \subset \overline{\Omega}_1$.
The unbounded exterior region $\mathbb{R}^m \setminus \overline{\Omega}_1$ does not include
the heterogeneity  and has a smooth boundary  $\Gamma$. 
It therefore supports spectrally accurate  BEM algorithms to simulate exterior scattered waves, and also exactly preserves the radiation condition (RC), even in the computational model.

In addition, the decomposition framework provides an analytical integral representation of the far-field using the scattered field, and hence 
our high-order FEM-BEM model {provides} relatively accurate approximations of the far-field arising
from the heterogeneous model. For inverse wave models,
accurate modeling of the far-field plays a crucial role in the  identification of   unknown 
wave propagation configuration properties   from far-field  measurements~\cite{bagheri,KressColton,gh:bayesian}.

Our approach in this article  is related to some  ideas presented in \cite{CeDomSay:2004,DomSay:2007,MR2292079}. 
The choice of two artificial boundaries leads to two bounded domains $\overline{\Omega}_0\subset\Omega_1\subset\Omega_2$ and  an overlapping  region  between $\Omega_1^{\rm c} =\mathbb{R}^m\setminus \overline{\Omega}_1$ and $\Omega_2$.  We prove that, under appropriate restrictions of
the scattered and absorbed fields in the overlapping region $\Omega_{12} (:= \Omega_1^{\rm c} \cap \Omega_2) $, our decomposed model is equivalent to the original Helmholtz model in the full space $\mathbb{R}^m$.
The  unknowns in our decomposed framework, which exactly incorporates the RC, are: (a) the trace of the scattered wave on $\Gamma$ 
that will yield  the solution in the unbounded domain $\Omega_1^{\rm c}$, through a boundary layer potential ansatz of the scattered field; (b) the trace of the total wave in the boundary $\Sigma$ of $\Omega_2$, that will provide the Dirichlet data to determine the total absorbed wave in the bounded domain 
$\Omega_2$. These properties will play a crucial role in designing and implementing  our high-order FEM-BEM  algorithm.

The FEM-BEM numerical algorithm can be  discerned at this point: It comprises approximating the absorbed wave field in a finite 
dimensional space using an FEM spline ansatz in the bounded domain $\Omega_2$, and by a BEM ansatz for the scattered field  in the unbounded region, 
exterior to $\Gamma$, and these  fields are constrained to (numerically) coincide on the overlapping domain $\Omega_{12}$, and hence on the
interface boundaries.  Since these artificial boundaries can be freely chosen, we can ensure  a bounded simple polygonal/polyhedral domain, more suitable for high-order FEM, and an unbounded region with a smooth boundary for spectrally accurate BEM.
In particular, the framework brings the best of the two numerical (FEM and BEM) worlds to compute the  fields accurately  for the
full heterogeneous model problem, without the need to truncate the unbounded  wave propagation region and approximate the RC.

The algorithmic construction and solving of the interface  linear system, which determines  key unknowns of the model on the interface boundaries (that is, the ansatz coefficients of the trace of the FEM and BEM solutions),   is challenging. 
However, important  properties of the continuous problem, such as a compact perturbation of the identity, are inherited by the numerical scheme. 
 \revA{Consequently, the  system of linear equations for the interface unknowns is very well conditioned}. 
 Such properties,  in conjunction with a cheaper matrix-vector multiplication for the underlying matrix, 
support the  use of  iterative solvers such as GMRES \cite{MR848568,MR1990645} to compute the ansatz coefficients. 
Major computational aspects of our high-order FEM and BEM discretizations in the framework
are independent and hence the underlying linear systems can be solved, {\em a priori}, by iterative Krylov methods.  
\revA{We show that the number of GMRES iterations, to solve the interface system, is   independent of various levels of discretization for a chosen frequency of the model.
For increasing frequencies, we also demonstrate that the growth of  the number of GMRES iterations is lower than the frequency growth.}

Instead  of using an iterative scheme for the interface system arising in our algorithm,  
one may also  consider the construction and storage of the matrix and 
a direct solver for the  system.
The advantage of the latter is that the interface problem matrix can be reused  for numerous incident input waves
that occur in many practical applications, for example,
to compute the monostatic cross sections, and also for developing appropriate  reduced order model (ROM)~\cite{tmatrom} versions of our algorithm. The  matrix arising {\revA{ in our}} interface system  is relatively small  because of the spectral accuracy of the BEM algorithm, and because  the system involves only unknowns on
the artificial interface boundaries. 
Hence  post-processing of the computed fields, such as for the evaluation  of the far-field, can be done quickly and efficiently.
The far-field output also plays a crucial role in developing  stable ROMs for wave propagation models~\cite{ghh:2011,tmatrom}. 

The paper is organized as follows. In Section~\ref{sec:decomposition} we present the  decomposition framework and  prove that, under very weak assumptions, the decomposition  is well-posed and is equivalent to the full heterogeneous and unbounded medium wave propagation model. In 
Section~\ref{sec:fem-bem} we present a numerical discretization for the two dimensional case, combining  high-order finite elements  with  spectrally accurate convergent boundary elements~\cite{Kress:2014} and describe the algebraic and  implementation details. 
In Section~\ref{sec:num-exp}  we demonstrate  the efficiency of the FEM-BEM algorithm for simulating wave propagation in  two distinct classes of (smooth and non-smooth) heterogenous media.

\section{Decomposition framework and well-posedness analysis}\label{sec:decomposition}
Let $\Omega_0\subset \mathbb{R}^m$, $m=2,3$, be a bounded  domain.  
The ratio of the speed of wave propagation inside  the heterogeneous (and not necessarily connected) region $\Omega_0$ and on its free-space exterior $\Omega_0^{\rm c} := \mathbb{R}^m \setminus \overline{\Omega}_0$  is described 
through a \revB{refractive index  function $n$} that we assume in this article to be piecewise smooth with  $1-n$ having compact support in $\overline{\Omega}_0$ (i.e, $n|_{\Omega_0^{\rm c}}\equiv 1$).}

The main focus of this article is to study the wave propagation in $\mathbb{R}^m$, induced by the impinging of an
incident wave $u^{\rm inc}$, say, a plane wave with wavenumber $k >0$.
More precisely, the continuous wave propagation model is to find 
the total field $u (:= u^s + u^{\rm inc }) \in H^1_{\rm loc}(\mathbb{R}^m)$  that satisfies the Helmholtz equation and
the Sommerfeld RC: 
\begin{equation}
 \label{eq:theproblem}
 \left|
 \begin{array}{rcl}
 \Delta u + k^2 n^2\:u &=&0,\quad \text{in }\mathbb{R}^m,\\
 \partial_r u^s-\mathrm{i}k{u}^s&=&o(|r|^{\frac{m+1}2}),\quad \text{as }|r|\to\infty.
 \end{array}
 \right.
\end{equation}
It is well known  that~\eqref{eq:theproblem} is uniquely solvable~\cite{Kress:2014}.
(Later in this section, we introduce the classical Sobolev spaces $H^s$, for $s \geq 0$, with appropriate norms.) 

\subsection{A decomposition framework} 

The heterogeneous-homogeneous model problem \eqref{eq:theproblem} is decomposed  by introducing  two 
artificial curves/surfaces $\Gamma$ and $\Sigma$ with interior $\Omega_1$ and $\Omega_2$ respectively satisfying $\overline{\Omega}_0\subset \Omega_1\subset \overline{\Omega}_1\subset\Omega_2$.  We assume from now on that $\Gamma$ is smooth and $\Sigma$ is a polygonal/polyhedral boundary. A sketch of the different domains is displayed in Figure \ref{fig:01}. Henceforth,
$\Omega_i^{\rm c}:=\mathbb{R}^m\setminus\overline{\Omega}_i, ~i = 0,1,2$. 

We introduce the following decomposed heterogeneous and homogeneous media auxiliary models: 
  \begin{itemize}
\item For  a given function $f_\Sigma^{\rm inp} \in H^{1/2}(\Sigma)$, we  seek  a propagating wave field $w$ so that $w$ and its trace  $\gamma_\Sigma w$ on the boundary $\Sigma$  satisfy 
\begin{equation}
 \label{eq:FEM:0}
 \left|
 \begin{array}{rcl}
 \Delta w + k^2 n^2\:w &=&0,\quad \text{in }\Omega_2,\\
 \gamma_\Sigma w  &=&f_\Sigma^{\rm inp}.
 \end{array}
 \right.
\end{equation}
Throughout the article, we assume  that  this interior problem is uniquely solvable. We introduce the following operator notation for the heterogeneous auxiliary   model:  For any Lipschitz $m$- or $(m-1)$-dimensional (domain or manifold)  $D\subset \Omega_2$, we define the solution operator
$\KSigmaD$ associated with the auxiliary model~\eqref{eq:FEM:0} as 
\begin{equation}
 \label{eq:FEM-oper:0}
\KSigmaD f_\Sigma^{\rm inp} :=w|_D. 
\end{equation}
 Two  cases will be of particular interest for us: $\KSigmaOmegaTwo \revB{f_\Sigma^{\rm inp}}$, which is nothing but $w$ satisfying~\eqref{eq:FEM:0}, and  
 $\KSigmaGamma f_\Sigma^{\rm inp}=\gamma_\Gamma w$, the trace of the solution ${w}$ of~\eqref{eq:FEM:0}  on $\Gamma{\subset\Omega_2}$. 
 
\item In the exterior unbounded  homogeneous medium $\Omega_1^{\rm c}:=\mathbb{R}^m\setminus\Omega_1$, for a given function  
$f_\Gamma^{\rm inp}{\in H^{1/2}(\Gamma)}$ we
seek a scattered field $\widetilde{\omega}$ satisfying 
\begin{equation}
 \label{eq:BEM:0}
 \left|
 \begin{array}{rcl}
 \Delta \widetilde{\omega} + k^2 \widetilde{\omega} &=&0,\quad \text{in }\Omega_1^{\rm c},\\
 \gamma_\Gamma \widetilde{\omega} &=&f_\Gamma^{\rm inp},\\
 \partial_r\widetilde{\omega}- {\rm i}k\widetilde{\omega} &=& o(|r|^{(m-1)/2}).
 \end{array}
 \right.
\end{equation} 
Unlike problem \eqref{eq:FEM:0}, \eqref{eq:BEM:0}  is always  uniquely solvable~\cite{Kress:2014}. We define the associated solution operator $\KGammaD$ as 
\begin{equation}
 \label{eq:BEM-oper:0}
\KGammaD f_\Gamma^{\rm inp} :=\widetilde{\omega}|_D, 
\end{equation}
with special attention to $\KGammaOmegaOneC f_\Gamma^{\rm inp}$ and $\KGammaSigma f_\Gamma^{\rm inp}$, namely the scattered
field $\widetilde{\omega}$ satisfying~\eqref{eq:BEM:0} and its trace $\gamma_\Sigma \widetilde{\omega}$. 
   \end{itemize}


The decomposition framework that we propose for the continuous problem  is the following:
\begin{subequations}
\label{eq:BEMFEM}
\begin{enumerate}
\item  Solve the interface boundary integral system  to find $(f_\Sigma,f_\Gamma)$, using  data 
$(\gamma_{\Sigma} u^{\rm inc}, \gamma_{\Gamma} u^{\rm inc})$ : 
\begin{equation}
 \label{eq:BEMFEM:0}
 \left|
 \begin{array}{ccccrcl}
 \multicolumn{4}{l}{ (f_\Sigma,f_\Gamma)\in H^{1/2}(\Sigma)\times {H^{1/2}(\Gamma)} }\\[1.1ex]
 f_\Sigma&-&\KGammaSigma f_\Gamma&=& \gamma_{\Sigma} u^{\rm inc}\\
 - \KSigmaGamma f_\Sigma&+&f_\Gamma &=& -\gamma_{\Gamma} u^{\rm inc}
 \end{array}
 \right.
\end{equation} 

\item Construct the total field for the model problem~\eqref{eq:theproblem} using the 
solution $(f_\Sigma,f_\Gamma)$ of~\eqref{eq:BEMFEM:0}, by solving the auxiliary models 
\eqref{eq:FEM:0} and \eqref{eq:BEM:0}:
\begin{equation}
 \label{eq:BEMFEM:01}
  u:=\begin{cases}
     \KSigmaOmegaTwo f_\Sigma,\quad&\text{in $\Omega_2$},\\
     \KGammaOmegaOneC f_\Gamma+u^{\rm inc},\quad&\text{in $\Omega_1^{\rm c}$}.
    \end{cases}~
 \end{equation}
\end{enumerate}
\end{subequations}

We claim that, provided  \eqref{eq:BEMFEM:0} is solvable, the decomposed framework-based 
field $u$ defined in~\eqref{eq:BEMFEM:01} is the solution of \eqref{eq:theproblem}. Notice that
we are implicitly assuming in  \eqref{eq:BEMFEM:01} that 
\begin{equation}\label{eq:2.6}
\KSigmaOmegaOneTwo f_\Sigma = u^{\rm inc}|_{\Omega_{12}}+\KGammaOmegaOneTwo  f_\Gamma,
 \end{equation}
where we recall the notation $\Omega_{12} =\Omega_1^{\rm c}\cap \Omega_2$.
Indeed, in view of \eqref{eq:BEMFEM:0}, both functions in \eqref{eq:2.6} agree on $\Sigma\cup \Gamma$  (the boundary of $\Omega_{12}$). 
Assuming, as we will do from now on, that the only solution to the homogeneous system
\begin{equation}
 \label{eq:BEMFEM:02}
 \left| 
 \begin{array}{rcl}
 \Delta v + k^2 v &=&0,\quad \text{in }\Omega_{12},\\
 \gamma_\Gamma v &=&0,\quad \gamma_\Sigma v\,=\,0
 \end{array}
 \right.
\end{equation} 
is the trivial one and noticing that $n|_{\Omega_{12}}\equiv 1$ which implies that $\KSigmaOmegaOneTwo  f_\Sigma $ and $\KGammaOmegaOneTwo f_\Gamma$  are solutions of the Helmholtz equation in $\Omega_{12}$, we can conclude  that \eqref{eq:2.6} holds. Since $u$ defined in \eqref{eq:BEMFEM:01} belongs to $H^1_{\rm loc}(\mathbb{R}^m)$, it is simple to check that this function is the solution of \eqref{eq:theproblem}.

We remark that the hypothesis we have taken on the artificial boundaries/domains, i.e. the well-posedness of problems \eqref{eq:FEM:0} and \eqref{eq:BEMFEM:02}, are not very restrictive in practice:  $\Sigma$ or $\Gamma$ can be modified if needed. Alternatively, one can consider different boundary conditions on $\Gamma$ and $\Sigma$ (such as Robin conditions), redefining $\KSigmaD$ and $\KGammaD$ accordingly,  which will lead to  a variant of the framework that we analyze in this article. In a future work we shall explore other boundary conditions on the interfaces and analysis of the resulting variant models.
\subsection{Well-posedness of the decomposed continuous problem}

The aim of this subsection is to prove that the system of equations \eqref{eq:BEMFEM:0}, under the above stated hypothesis, has a  unique solution. Consequently,  we can conclude  that  the decomposition for the exact solution presented in \eqref{eq:BEMFEM:01} exists and  is unique.   To this end, we first derive  some regularity results related to the  operators  $\KSigmaD$ and $\KGammaD$ in Sobolev spaces. 
\revB{For  the topic of Sobolev spaces,  we refer the reader to~\cite{AdFo:2003,McLean:2000}.}


\subsubsection{Functional spaces} 
Let $D\subset \mathbb{R}^m$ be a Lipschitz domain. For any non-negative integer $s$, we denote
\[
 \|f\|_{H^s(D)}^2:= \sum_{|\bm{\alpha}|\le s} \int_D |\partial_{\bm{\alpha}} f|^2
\]
the Sobolev norm, where the summation uses the  standard multi-index notation in $\mathbb{R}^m$. 
For $s= s_0+\beta$ with $s_0$ a non-negative integer and $\beta\in (0,1)$, we set 
\[
 \|f\|_{H^s(D)}^2:=  \|f\|_{H^{s_0}(D)}^2 + \sum_{|\bm{\alpha}|\le  s_0} \int_D\int_D 
 \frac{|\partial_{\bm{\alpha}} f({\bf x})-\partial_{\bm{\alpha}} f({\bf y})|^2 }{|{\bf x}-{\bf y}|^{m+2\beta}}{{\rm d}{\bf x}\,{\rm d}{\bf y}}. 
\]
The Sobolev space $H^s(\Omega)$ ($s\ge 0$) can be defined as,
\[
 H^s(D):=\{f\in L^2(D)\: : \: \|f\|_{H^s(D)}<\infty \},
\]
endowed with the above natural norm. 

If $\partial D$ denotes the boundary of $D$, we can introduce $H^s(\partial D)$ with a similar construction using local charts: Let $\{\partial D^j,\mu^j,{\bf x}^j\}_{j=1}^{J}$ be an atlas of $\partial D$, that is, $\{\partial D\}_j$  is an open covering of $\partial D$,    $\{\mu^j\}$ a subordinated Lipschitz partition of unity on $\partial D$, and ${\bf x}^j:\mathbb{R}^{m-1}\to \partial D $ being  Lipschitz and injective with $\partial D^j\subset \mathop{\rm Im} {\bf x}^j$, then we define 
\[
 \|\varphi\|_{H^s(\partial D)}^2:=\sum_{j=1}^J \|(\mu^j \varphi)\circ {\bf x}^j \|^2_{H^s(\mathbb{R}^{m-1})}.
\]
We note that $(\mu^j \varphi)\circ{\bf x}^j$ can be extended by zero outside of the image of ${\bf x}^j$. We then set
\[
 H^s(\partial D):=\{\varphi\in L^2(\partial D) \  : \ \|\varphi\|_{H^s(\partial D)}<\infty\}.
\]
The space $H^s(\partial D)$ is well defined for $s\in[0,1]$: Any  choice of $\{\partial D^j,\mu^j,{\bf x}^j\}$ gives rise to an equivalent norm (and inner product). 
If $\partial D$ is a ${\cal C}^m$-boundary, 
such as   $\Gamma$ in Figure~\ref{fig:01},  this construction can be set up for $s\in[0,m]$ by taking $\{{\bf x}^j,{\omega}^j\}$ to be in ${\cal C}^m$ as well. In particular, if $\partial D$ is smooth we can define $H^s(\partial D)$ for any $s\ge 0$.  Further, the space $H^{-s}(\partial D)$ can be defined as the realization of the dual space of $H^{s}(\partial D)$ when the integral product is taken as  a representation of the duality pairing.

It is a classical result that the trace operator $\gamma_{\partial D} u:= u|_{\partial D}$ \revB{defines} a continuous onto  mapping from $H^{s+1/2}(D)$  into $H^s(\partial D)$ for any $s\in (0,1)$. Actually, if $\partial D$ is smooth then $s\in(0,\infty)$. In these cases, we can alternatively define
\[
 H^s(\partial D): =\{\gamma_{\partial D} u\ : \ u\in H^{s+1/2}(D)\}
\]
endowed with the image norm:
\begin{equation}
\label{eq:ImNorm}
\|\varphi\|_{H^s(\partial D)}:= \inf_{{0\ne u\in H^{s+1/2}(D)\atop \gamma_\Sigma u = \varphi}} \|u\|_{H^{s+1/2}(D)}.
\end{equation}
We will use this definition to extend  $H^{s}(\partial D)$ for $s>1$ in the  Lipschitz  case. Notice that with this definition, the trace operator from $H^{s+1/2}(D)$  into $H^s(\partial D)$ is continuous  for any $s>0$. 
%
%
%
%

\subsubsection{Boundary potentials and integral operators}

Let $\Phi_k$ be  the fundamental solution for the two- or three-dimensional constant coefficient  Helmholtz 
operator  ($\Delta +k^2I$) equation, defined for $\mathbf{x}, \mathbf{y}  \in \mathbb{R}^m$ with $r := |\mathbf{x}-\mathbf{y}|$ as
\begin{equation}
  \label{eq:fund}
 \Phi_k(\mathbf{x}, \mathbf{y}) :=\begin{cases}
 \displaystyle
  \frac{\mathrm{i}}{4} H^{(1)}_0(kr),  &  \mathbf{x}, \mathbf{y} \in \mathbb{R}^2, \\ 
 \displaystyle \frac{1}{4 \pi r} \exp(\mathrm{i}kr), &   \mathbf{x}, \mathbf{y} \in \mathbb{R}^3,
  \end{cases}
\end{equation}
where $H^{(1)}_n$ denotes the first kind Hankel
function of order  $n$. For a smooth curve/surface $\Gamma$, with outward unit normal ${\boldsymbol \nu}$ and
normal derivative at ${\bf y} \in \Gamma$ denoted by $\partial_{{\boldsymbol \nu}({\bf y})}$, let 
\[
({\rm SL}_k\varphi)({\bf x}) :=\int_{\Gamma} \Phi_k({\bf x}-{\bf y})\varphi({\bf y}) \,{\rm d}\sigma_{\bf y},\quad
 ({\rm DL}_k g)({\bf x}) :=\int_{\Gamma} \partial_{{\boldsymbol \nu}({\bf y})}\Phi_k({\bf x}-{\bf y})g({\bf y}) \,{\rm d}\sigma_{\bf y}
, \qquad {\bf x}\in \mathbb{R}^m\setminus\Gamma,\]
denote the single- and double-layer potentials,  with density functions $\varphi$ and $g$, respectively. 

The single- and double-layer  boundary integral operators are then given, 
via the well-known jump relations~\cite{KressColton} for the boundary layer potentials, by 
\begin{eqnarray}\label{eq:Vk}
 \mathrm{V}_{k}\varphi &:=&(\gamma_{\Gamma}{\rm SL}_k)\varphi=\int_{\Gamma} \Phi_k(\,\cdot\,-{\bf y})\varphi({\bf y}) \,{\rm d}\sigma_{\bf y},
 \\
 \mathrm{K}_{k}g &:=&\pm\tfrac12 g + (\gamma^{\mp}_{\Gamma}{\rm DL}_k)g
 =\int_{\Gamma} \partial_{{\boldsymbol \nu}({\bf y})}\Phi_k(\,\cdot\,-{\bf y})g({\bf y})  \,{\rm d}\sigma_{\bf y},
 \label{eq:Dk}
\end{eqnarray}
where  $\gamma^-_{\Gamma}$ and  $\gamma^+_{\Gamma}$ are trace operators on $\Gamma$, respectively, from the interior $\Omega_1$ and 
exterior $\Omega_1^{\rm c}$. Given a real non-vanishing   smooth function $\sigma:\Gamma\to \mathbb{R}$, and 
$\mathrm{V}_{k,\sigma}\phi :=  \mathrm{V}_k(\sigma \phi)$  for any $\phi \in   H^s(\Gamma)$, we consider the 
combined field acoustic  layer operator 
\begin{equation}
\label{eq:2.6a}
 \tfrac12 {\rm I}+\mathrm{K}_k -\mathrm{i}k\mathrm{V}_{k,\sigma}:H^s(\Gamma)\to H^s(\Gamma).
\end{equation}
\revB{Throughout this article,  ${\rm I}$ denotes the identity operator.}
The standard combined field operator used in the literature~\cite{KressColton} is based on the choice $\sigma\equiv 1$. 
In this article,  we do not restrict ourselves to the usual choice for reasons which will be fully explained later. Since
$\Gamma$ is smooth and that 
 $\mathrm{K}_k, \mathrm{V}_{k,\sigma} :H^s(\Gamma)\to H^{s+1}(\Gamma)$ are continuous, the operator in \eqref{eq:2.6a} is invertible as a consequence of the Fredholm alternative and the injectivity of \eqref{eq:2.6a}, which follows from a very simple modification of  the classical argument in \cite[Th 3.33]{KressColton}).

Thus the inverse of the combined field integral operator
\begin{equation}
\label{eq:3.4}
{\cal L} _{k,\sigma}:=\Big(\tfrac12 {\rm I}+\mathrm{K}_k -\mathrm{i}k\,\mathrm{V}_{k,\sigma}\Big)^{-1}{:H^s(\Gamma)\to H^s(\Gamma)}
\end{equation}
is well defined. 
Further, using \eqref{eq:Vk}-\eqref{eq:Dk} and with $\mathrm{SL}_{k,\sigma}\phi := \mathrm{SL}_k(\sigma\phi)$ for any $\phi \in   H^s(\Gamma)$, we can write
the solution operator occurring in the construction~\eqref{eq:BEMFEM:01} as
\begin{equation}
\label{eq:KGammaOmega}
\KGammaOmegaOneC=(\mathrm{DL}_k-\mathrm{i}{k}\,\mathrm{SL}_{k,\sigma}) {\cal L}_{k, \sigma}.
\end{equation}
The above solution operator, a variant of the Brakhage-Werner formulation (BWF)~\cite{MR0190518, KressColton}, will be used  in this article 
for both theoretical and computational purposes. The choice  $\sigma\equiv 1$ reduces to the standard BWF~\cite{MR0190518, KressColton}. 

\subsubsection{Well-posedness analysis  of the interface model}
In this subsection, we first develop two key results before proving well-posedness of the boundary integral system~\eqref{eq:BEMFEM:0}.
\begin{lemma}\label{lemma:01}
The operator 
\begin{equation}\label{eq:2.8}
\KGammaOmegaOneC: H^{s}(\Gamma)\to H_{\rm loc}^{s+1/2}(\Omega_1^{\rm c})
\end{equation}
is continuous for any $s\in [0,\infty)$. Further, for any bounded Lipschitz domain/manifold  $D\subset \Omega_1^{\rm c}$ with $\overline{D}\cap \Gamma =\emptyset$, the solution operator $\KGammaD$ in ~\eqref{eq:BEM-oper:0},
for the homogeneous media problem~\eqref{eq:BEM:0},
satisfies the following mapping property for any $s,r\ge 0$ 
\begin{equation}\label{eq:2.85}
\KGammaD: H^{s}(\Gamma)\to H^{r}(D).
\end{equation}
In particular, 
\begin{equation}\label{eq:2.85b}
 \KGammaSigma: H^{s}(\Gamma)\to H^{r}(\Sigma)
\end{equation}
is continuous and compact, for $s,r\in\mathbb{R}$. 
\end{lemma}
\begin{proof}
 The first desired property follows from the identities \eqref{eq:KGammaOmega}, \eqref{eq:3.4} and the well known mapping properties 
\begin{equation}\label{eq:2.75}
 \mathrm{DL}_k:H^s(\Gamma)\to H_{\rm loc}^{s+1/2}(\Omega_1^{\rm c}),\quad
 \mathrm{SL}_k:H^{s-1}(\Gamma)\to H_{\rm loc}^{s+1/2}(\Omega_1^{\rm c}),
\end{equation}
see for instance \cite[Th. 6.12]{McLean:2000}.
If   $\overline{D}\cap \Gamma =\emptyset$, 
 the kernels in the  boundary potentials in $\mathrm{DL}_k$ and $\mathrm{SL}_k $ are smooth functions in $\overline{D}\times \Gamma$ and hence the properties \eqref{eq:2.85} and \eqref{eq:2.85b}  hold.
\end{proof}

Next we consider the heterogeneous media model solution operator  $\KSigmaOmegaTwo$, as defined 
in~\eqref{eq:FEM:0}-~\eqref{eq:FEM-oper:0}. We recall the well known classical estimate~\cite{Ihlenburg:1998}
\[
 \|\KSigmaOmegaTwo f_\Sigma^{\rm inp}\|_{H^1(\Omega_2)}\le C\|f_\Sigma^{\rm inp}\|_{H^{1/2}(\Sigma)},
\]
with $C>0$ being a constant independent of $\revB{f_\Sigma^{\rm inp}}$.  
Below, we generalize this to obtain a higher regularity, using  boundary layer potentials and boundary integral operators, {defined in this case on barely Lipschitz curves/surfaces}  to improve the estimate for domains $D$ with $\overline{D}\subset \Omega_2\setminus\overline{\Omega}_1$. 

\begin{lemma}\label{lemma:02}
 There exists \revB{a constant} $C=C(k,n,\Omega_2)$ so that for any $s\in[0,1]$ and $f_\Sigma^{\rm inp}\in{H^ {s}(\Sigma)}$, 
\begin{equation}\label{eq:2.10}
 \|\KSigmaOmegaTwo  f_\Sigma^{\rm inp}\|_{H^ {s+1/2}(\Omega_2)}\le C \|f_\Sigma^{\rm inp}\|_{H^ {s}(\Sigma)}.
\end{equation}
Furthermore, if $D\subset \overline{D}\subset \Omega_2\setminus\overline{\Omega}_1$  the following solution
operator  mapping property holds for any $r\in \mathbb{R}$
\begin{equation}\label{eq:2.95}
 \KSigmaD : H^{0}(\Sigma)\to H^{r}(D).
\end{equation}
Consequently,  
\begin{equation}\label{eq:2.95b}
 \KSigmaGamma : H^{0}(\Sigma)\to H^{r}(\Gamma)
\end{equation}
is continuous and compact,  for any $r\in\mathbb{R}$. 
\end{lemma}
\begin{proof} {Throughout  this proof we let  $s\in[0,1]$ and, for notational convenience, we denote $v:=\KSigmaOmegaTwo f_\Sigma^{\rm inp}$}.
 Since, by definition, 
\[
 \Delta v+k^2 v= k^2(1-n^2)v,\quad \gamma_{\Sigma}v=f_\Sigma^{\rm inp}. 
\]
By the third Green identity (see for instance \cite[Th. 6.10]{McLean:2000}) we have the representation  
\begin{equation}\label{eq:2.9}
v =  k^2\int_{\Omega_0} \Phi_k(\cdot-{\bf y}) g^v_n({\bf y})\,{\rm d}{\bf y}+ 
 \mathrm{SL}_{k,\Sigma}\lambda^v_\Sigma-\mathrm{DL}_{k,\Sigma}f_\Sigma^{\rm inp},
\end{equation}
with $\mathop{\rm supp} {g^v_n}\subset\Omega_0$, where {we have used} the notation 
\[
 \lambda^v_\Sigma:= \partial_{\boldsymbol \nu} v,\quad g^v_n:=(1-n^2)v.
\] 
In the expression above $\mathrm{SL}_{k,\Sigma}$ and $\mathrm{DL}_{k,\Sigma}$ denote respectively the single- and double-layer potential from the corresponding densities, defined on $\Sigma$, associated with  the 
constant coefficient Helmholtz operator $\Delta +k^2{\rm I}$. Next we prove  that
 \[
  \|\lambda^v_\Sigma\|_{H^ {s-1}(\Sigma)}=\|\partial_{\boldsymbol \nu} v\|_{H^ {s-1}(\Sigma)}\le C \|f_\Sigma^{\rm inp}\|_{H^ {s}(\Sigma)}. 
  \]
To this end, we start from the decomposition  $v=v_1+v_2$, where the harmonic $v_1$   and the interior wave-field $v_2$ are  solutions of 
\[
 \left|\begin{array}{l}
        \Delta v_1 = 0,\quad \text{in }\Omega_2,\\
        \gamma_\Sigma v_1 =f_\Sigma^{\rm inp},
       \end{array}\right.\quad\text{and} \quad
 \left|\begin{array}{l}
        \Delta v_2+k^ 2 n^2 v_2 = -k^ 2 n^2 v_1,\quad \text{in }\Omega_2,\\
        \gamma_\Sigma v_2 =0.
       \end{array} \right.
\]
Classical  potential theory results, see   \cite[Th 6.12]{McLean:2000} and  the discussion which follows it {(see also  references therein)}, show that there exists $C>0$ so that  
\begin{equation}\label{eq:2.11}
  \|v_1\|_{H^ {s+1/2}(\Omega_2)}\le C \|f_\Sigma^{\rm inp}\|_{H^{s}(\Sigma)},\quad \|\partial_{{\bm \nu}} v_1\|_{H^ {s-1}(\Sigma)}\le C' \|f_\Sigma^{\rm inp}\|_{H^{s}(\Sigma)},
\end{equation}
for any $f_\Sigma^{\rm inp}\in H^s(\Sigma)$. On the other hand, following ~\cite[Ch. 4]{Gri:1985} or \cite{Dauge:1988}
there exists $\varepsilon>0$ and $C_{\varepsilon}>0$ such that  
\begin{equation}\label{eq:2.12}
 \|v_2\|_{H^{3/2+\varepsilon}(\Omega_2)}\le C_\varepsilon\|v_1\|_{H^0(\Omega)}\le C_\varepsilon\|f_\Sigma^{\rm inp}\|_{H^{0}(\Sigma)}.
\end{equation}
By the trace theorem (applied to $\nabla v_2$), 
\[
 \|\partial_{\boldsymbol \nu} v_2\|_{H^ {0}(\Sigma)} \le C\|\nabla v_2\|_{H^{1/2+\varepsilon}(\Omega_2)} \le 
 C'\|v_2\|_{H^{3/2+\varepsilon}(\Omega_2)}\le C''\|f_\Sigma^{\rm inp}\|_{H^{0}(\Sigma)}.
\]
Combining these estimates with \eqref{eq:2.9} we conclude that 
\begin{eqnarray*}
\|v\|_{H^{s+1/2}(\Omega_2)}&\le& C_s \big(\|g_n^v\|_{L^2(\Omega_0)} + \|\lambda^v_\Sigma\|_{H^{s-1}(\Sigma)}+\|f_\Sigma^{\rm inp}\|_{H^{s}(\Sigma) }\big) \\
 &\le& C'_s \big(\|v\|_{L^2(\Omega_0)} +\|f_\Sigma^{\rm inp}\|_{H^{s}(\Sigma)} \big)  \\
 &\le& C''_s \|f_\Sigma^{\rm inp}\|_{H^{s}(\Sigma)}. 
\end{eqnarray*}
Notice also that if $D\subset \overline{D}\subset \Omega_2\setminus\overline{\Omega}_1$, 
because the kernels of the potentials operators and the Newton potential are smooth in the corresponding variables,
we gain from the extra smoothing properties of the underlying operators  in \eqref{eq:2.9}   to derive
\[
 \|v\|_{H^{r}(D)} \le C \big(\|g^v\|_{L^2(\Omega_0)} + \|\lambda^v_\Sigma\|_{H^{-1}(\Sigma)}+\|f_\Sigma^{\rm inp}\|_{H^{0}(\Sigma)} \big)\le 
 C'\|f_\Sigma^{\rm inp}\|_{H^0(\Sigma)}, 
\] 
\revB{where the constants $C$ and $C'$ are independent of $f_\Sigma^{\rm inp}$.}
%

 \end{proof}
%
For deriving the main desired result of this section, it is convenient to  define the following  off-diagonal 
operator matrix
\[
 {\cal K}:=\begin{bmatrix}
                             &\KSigmaGamma\\
            \KGammaSigma &
          \end{bmatrix}.
\]
Then \eqref{eq:BEMFEM:0} can be written in operator form 
\begin{equation}
 \label{eq:BEMFEM:0b}
 \left({\cal I}-{\cal K}\right)\begin{bmatrix}
                          f_\Sigma\\
                          f_\Gamma
                          \end{bmatrix}=\begin{bmatrix*}[r]
                          \gamma_{\Sigma} u^{\rm inc}\\
                          -\gamma_{\Gamma} u^{\rm inc}
                          \end{bmatrix*},
\end{equation}
\revB{where ${\cal I}$ denotes the  $2\times 2$ block identity operator.}
 A simple consequence of Lemmas \ref{lemma:01} and \ref{lemma:02} is that
 \[
 {\cal I}-{\cal K}:H^s(\Sigma)\times H^{\revAB{s}}(\Gamma)\to H^s(\Sigma)\times H^s(\Gamma)
 \]
 is continuous for any $s\ge 0$. Next we prove  that  this operator is indeed an isomorphism: 
%
%

\begin{theorem}\label{th:3.3} For any $s \ge 0$,
\[
 {\cal I}-{\cal K}: H^s(\Sigma)\times H^s(\Gamma)\to H^s(\Sigma)\times H^s(\Gamma),
\]
is an invertible compact perturbation of the identity operator. 
\end{theorem} 
\begin{proof} The continuity of  ${\cal K}: H^0(\Sigma)\times H^0(\Gamma)\to H^s(\Sigma)\times H^s(\Gamma)$ for any $s \ge 0$ has already been established in the   two preceding lemmas. In particular, ${\cal K}$ is compact. Moreover, the null space ${\cal I}-{\cal K}$ consists of smooth functions. For any $( g_\Sigma,g_\Gamma)\in N({\cal I}-{\cal K})$,
we construct
 \[
  v :=\KSigmaOmegaTwo g_\Sigma,\quad
  \vartheta :=\KGammaOmegaOneC g_\Gamma. 
 \]
 Note that $w:=(v-\vartheta) $ defined, in principle, in $\Omega_{12}={\Omega_2\cap \Omega_1^{\rm c}}$ satisfies
 \[
  \Delta w+k^ 2w=0,\quad \text{in } \Omega_{12},\quad \gamma_\Sigma w=\gamma_{\Gamma} w=0.  
 \]
 By the well-posedness of problem \eqref{eq:BEMFEM:02}, we have  $w=0$ in $\Omega_{12}$. We define $u$ on
 $\mathbb{R}^m$ as 
 \[
  u({\bf x}) = \begin{cases}
       v({\bf x}),\quad&\text{if }{\bf x} \in \Omega_2,\\
       \vartheta({\bf x}),\quad&\text{if }{\bf x}\in \Omega_1^{\rm c}.\\
      \end{cases}
 \]
  Note that $u$ is well defined in $\Omega_{12}$, {and} it is  a solution of \eqref{eq:theproblem} with incident wave {$u^{\rm inc}=0$}. Therefore, $u=0$ which implies that $\vartheta=0$ in $\Omega_2^{\rm c}$. The principle of analytic continuation yields that $\vartheta=0$ also in $\Omega_1^{\rm c}$ and therefore $g_\Gamma =\gamma_{\Gamma} \vartheta=0$. Finally,
\[
 g_{\Sigma} =\gamma_{\Sigma } u=\gamma_{\Sigma }\vartheta = 0, 
\]
and hence the desired result follows. 
\end{proof}

\section{A FEM-BEM algorithm for the decomposed model}\label{sec:fem-bem}

In this section we consider the numerical discretizations on the proven equivalent decomposed system~\eqref{eq:BEMFEM}.
In this article, we restrict  to the 
two-dimensional (2-D) case. [The 3-D algorithms and analysis for~\eqref{eq:BEMFEM} will be different  to 
the 2-D case, and in a future work we shall investigate a 3-D  FEM-BEM computational model.]
Briefly, the approach consists of replacing  the continuous operators $\KSigmaOmegaTwo$ and $\KGammaD$ with suitable
high-order  FEM and BEM procedures-based discrete operators. The stability of such a discretization depends on the numerical methods chosen in each case. 

For discretization of  the differential  operator  $\KSigmaOmegaTwo$ based on the  heterogeneous domain model, we could consider a standard FEM with triangular,   quadrilateral  or even more complex  elements. We will choose the first case, for the sake of simplicity, and we expect 
the analysis developed in this case could cover these other types of elements, with appropriate minor modifications. 

The BEM procedure, for discretizing the exterior homogeneous medium associated $\KGammaD$ through boundary integral
operators,  is more open since an extensive range of methods is available in the literature. We will restrict ourselves to the spectral Nystr\"om method~\cite{Kress:2014} {(see also \cite{MR3526814})}. This scheme provides a discretization of the four integral operators of the associated Calderon calculus, and has exponential rate of convergence. In this article, we will make use of high-order discretizations of the \revB{single- and double-layer operators} that are easy to implement.

A key restriction  of the standard Nystr\"om method to achieve spectrally accurate convergence is  the requirement \revB{of a} smooth diffeomorphic parameterization of the boundary. This is because the method starts from appropriate decompositions and factorizations of the kernels of the operators to split 
these into regular and singular parts. This is not a severe restriction in our case since $\Gamma$ is an auxiliary user-chosen smooth curve and can therefore  be  easily constructed as detailed as required.

Next we briefly consider these two known numerical procedures and hence describe  our combined FEM-BEM algorithm and implementation details.

\subsection{The FEM procedure }
Let $\{\mathcal{T}_h\}_h$ be a sequence of regular triangular meshes where $h$ is the discrete mesh parameter,  the diameter of the largest element of the grid. Hence we write  $h\to 0$  to mean that the maximum of the diameters of the elements tends to 0. 
Using ${\cal T}_h$, we  construct the finite dimensional spline  approximation space 
\[
  \mathbb{P}_{h,d}:=\{{v}_h\in{\cal C}^0(\Omega_2) :\ v_h|_{T}\in\mathbb{P}_d, \ \forall T\in {\cal T}_h \}, 
 \]
where $\mathbb{P}_d$ is the space of \revB{bivariate polynomials} 
of degree $d$. We define  the  FEM approximation   $\KSigmaOmegaTwo^h$  to $\KSigmaOmegaTwo$  as follows:
The FEM operator
 \[
\KSigmaOmegaTwo^h:\gamma_\Sigma \mathbb{P}_{h,d}\to  \mathbb{P}_{h,d},
 \]  
for  $f_{\Sigma,h}^{\rm inp} \in  \gamma_\Sigma \mathbb{P}_{h,d}$,
 is constructed as $u_h:=\KSigmaOmegaTwo^hf_{\Sigma,h}^{\rm inp}$, where $u_h \in  \mathbb{P}_{h,d}$ is  the solution of the  discrete FEM equations:
 \begin{equation}\label{eq:4.1}
  \left|
  \begin{array}{l}
b_{k,n}(u_h,v_h)=0,\quad \forall v_h\in \mathbb{P}_{h,d}\cap H_0^1(\Omega_2)\\
\gamma_\Sigma u_h = f_{\Sigma,h}^{\rm inp},
  \end{array},
  \right. \qquad b_{k,n}(u,v)=\int_{\Omega_2}\nabla u \cdot \overline{\nabla v} - k^2 \int_{\Omega_2} n^2\, u  \overline{v}. 
 \end{equation}
The discrete FEM {operator $\KSigmaOmegaTwo^h$ is well defined for sufficiently small $h$}.

\subsection{The BEM procedure }

Let 
\begin{equation}\label{eq:parameterization}
{\bf x}:\mathbb{R}\to\Gamma, \qquad {\bf x}(t) :=(x_1(t),x_2(t)), \quad t \in \mathbb{R}
\end{equation}
be a smooth $2\pi-$periodic regular parameterization of $\Gamma$. 
We denote  by the same symbol ${\rm SL}_k$, ${\rm DL}_k$, ${\rm V}_k$ and ${\rm K}_k$ the parameterized layer potentials and boundary layer operators:
\begin{eqnarray*}
 ({\rm SL}_k\varphi)({{\bf z}})&=&\int_0^{2\pi} \Phi_k({{\bf z}}-{\bf x}(t))\varphi(t)\,{\rm d}t\,\quad\\
 ({\rm DL}_kg)({{\bf z}})&=&\int_0^{2\pi} \big(\nabla_{\bf y} \Phi_k({{\bf z}}-{\bf y})\big)\Big|_{\bf y={\bf x}(t)}\cdot \bm{\mu}(t)\,g(t)\,{\rm d}t
 \end{eqnarray*}
where $\bm{\mu}(t):=(x_2'(t),-x_1'(t))=|{\bf x}'(t)|\:{\bm \nu}\circ{\bf x}(t)$. 
Observe that $|{\bf x}'(t)|$ is incorporated \revB{into} the density in ${\rm SL}_k$ and to the kernel in ${\rm DL}_k$. We follow the same convention for the single- and double-layer weakly singular  boundary integral operators. For high-order approximations,  
it is important to efficiently take care of  the singularities. In particular, for the spectrally accurate
Nystr\"om  BEM solver, we use the following representations  of the layer operators with smooth
$2\pi$ bi-periodic kernels  $A,\ B,\ C,\ D$~\cite{KressColton}:
\begin{eqnarray*}
 ({\rm V}_k\varphi)(s)&=&\int_{0}^{2\pi}A(s,t)\log\sin^2\tfrac{s-t}2\:\varphi(t)\:{\rm d}t+\int_{0}^{2\pi}B(s,t)\varphi(t)\:{\rm d}t,\\ 
 ({\rm K}_kg)(s)&=&\int_{0}^{2\pi}C(s,t)\log\sin^2\tfrac{s-t}2\:g(t)\:{\rm d}t+\int_{0}^{2\pi}D(s,t)g(t)\:{\rm d}t.
 \end{eqnarray*}

The Nystr\"om method, based on a discrete positive integer parameter  $N$,  starts with setting up a uniform grid  
\begin{equation}\label{eq:grid-points}
 t_j := {\tfrac{\pi j}{N}},\quad j = -N+1, \dots, N,
 \end{equation}
 {and the space of trigonometric polynomials of degree at most $N$}
   \begin{equation}\label{eq:def:Tn}
  \mathbb{T}_N:= \spann\langle {\exp({\rm i}\ell t)}\ :\ \ell\in\mathbb{Z}_N\rangle, 
 \end{equation}
 with $\mathbb{Z}_N:= \{-N+1,-N+2,\ldots, N\}.$
 We next introduce the interpolation operator  $\mathrm{Q}_N$
 \begin{equation}\label{eq:interp}
\mathbb{T}_N\ni  \mathrm{Q}_N\varphi \quad \text{s.t.}\quad (\mathrm{Q}_N\varphi)(t_j)=\varphi(t_j),
\qquad j = -N+1, \dots, N,
 \end{equation}
to define  discretizations of the single and double layer operators:
\begin{eqnarray*}
 ({\rm V}_{k}^{N}\varphi)(s)&:=& \int_{0}^{2\pi}{\rm Q}_N(A(s,\cdot)\varphi\big)(t)\log\sin^2\tfrac{s-t}2\:{\rm d}t+\int_{0}^{2\pi}{\rm Q}_N(B(s,\cdot)\varphi\big)(t)\:{\rm d}t,\\ 
 ({\rm K}_{k}^{N}g)(s)      &:=& \int_{0}^{2\pi} {\rm Q}_N(C(s,\cdot)g\big)(t)\log\sin^2\tfrac{s-t}2\:{\rm d}t+\int_{0}^{2\pi}{\rm Q}_N(D(s,\cdot)g\big)(t)\:{\rm d}t.
 \end{eqnarray*}
 We stress that the above integrals can be computed exactly using  the identities:
 \[
-\frac{1}{2\pi}\int_{0}^{2\pi} \log\sin^2\tfrac{t}2\:{\exp({\rm i}\ell t)}\:{\rm d}t = 
-\frac{1}{2\pi} \int_{0}^{2\pi} \log\sin^2\tfrac{t}2 \cos({\ell} t)\:{\rm d}t=\begin{cases}
                                                          \log 4     , &{\ell}=0,\\
                                                          \frac{1}{|{\ell}|}, &{\ell}\ne 0,
                                                         \end{cases}
 \]
 and for $g_N \in \mathbb{T}_N$, 
 \begin{equation}\label{eq:quad}
   \int_{0}^{2\pi} g_N(t)\:{\rm d}t =
 \frac{\pi}{N} \sum_{j=0}^{2N-1} g_N(t_j),
 %
 \end{equation}
which are based on   properties of the trapezoidal/rectangular rule for $2\pi-$periodic functions.

 The high-order approximation evaluation of the potentials is achieved in a similar way:
 \begin{equation}\label{eq:SLNDLN}
 \begin{aligned} 
 \big({\rm SL}_{k}^{N}\varphi\big)({{\bf z}})\ &:=\ \int_{0}^{2\pi}
 {\rm Q}_N(\Phi_k({{\bf z}}-{\bf x}(\cdot))\varphi)(t)\,{\rm d}t, \\ 
 \big({\rm DL}_{k}^{N}g\big)({{\bf z}})\ &:=\ \int_{0}^{2\pi}
 {\rm Q}_N\big(\big(\nabla_{\bf y}\Phi_k({{\bf z}}-{\bf y})\big) \big|_{\bf y={\bf x}(\cdot)}
  \cdot \bm{\nu}(\cdot)\,g\big)
  (t)\,{\rm d}t,
  \end{aligned} 
 \end{equation}
{leading to the rectangular rule approximation as in \eqref{eq:quad}}

Now we are ready to describe  the discrete operator $\KGammaOmegaOneC^N$ that is 
 a high-order  approximation to 
the exterior homogeneous model continuous operator $\KGammaOmegaOneC$.  First, we introduce 
the parameterized counterpart of  the continuous operator in  \eqref{eq:2.6a}, 
\begin{equation}\label{eq:4.2}
{\cal L}_k g :=( \tfrac12\mathrm{I}+\mathrm{K}_k-\mathrm{i}k\mathrm{V}_k)^{-1} g, 
\end{equation} 
(which corresponds to $\sigma\circ{\bf x} = \frac{1}{|{\bf x}|}$). Then  we define 
 \begin{equation}\label{eq:00} 
  \KGammaOmegaOneC^{N}
  g:=(\mathrm{DL}_{k}^{N}-{\rm i}k\mathrm{SL}_{k}^{N})  {\cal L}_{k}^{N}  g,\quad \text{with }
 {\cal L}_{k}^{N}:= (\tfrac12\mathrm{I}+\mathrm{K}_{k}^{N}-\mathrm{i}k\mathrm{V}_{k}^{N})^{-1}.
 \end{equation}

We remark  that the definition of $\KGammaOmegaOneC^N$ requires only evaluation of input functions at the grid points. {In particular, it is well defined on continuous functions.}
Indeed, we have
 \[
 \varphi = {\cal L}_{k}^{N}g 
 \quad\Rightarrow\quad
 {\rm Q}_N\varphi ={\rm Q}_N{\cal L}_{k}^{N}{\rm Q}_N g, 
 \]
 and since the  discrete boundary layer operators only use pointwise values of the density at the grid points (i.e., {${\rm Q}_N\varphi$}), 
 evaluation of $\KGammaOmegaOneC^{N}g$
  requires only values of $g$    at the grid points. So we can  replace, when necessary,   
 \begin{equation}\label{eq:3.10}
    \KGammaOmegaOneC^N g  =  {\KGammaOmegaOneC^N {\mathrm Q}_N g.}
 \end{equation}
The discrete operator  $\KGammaSigma^N g$ {is defined accordingly} by taking the trace of $\KGammaOmegaOneC^N g$ on  $\Sigma$.
Thus our algorithm is based on the idea of taking the trace of FEM and BEM solutions on $\Gamma$ and $\Sigma$ respectively.


\subsection{The FEM-BEM computational model}
In addition to the discrete operators defined above, we need one {last} discrete 
operator to describe the {FEM-BEM algorithm. Let}
\begin{equation}
 {\rm Q}^h_{\Sigma}:{\cal C}^0(\Sigma)\to \gamma_\Sigma \mathbb{P}_{h,d},
\end{equation}
 {denote} the usual Lagrange interpolation operator on $\gamma_\Sigma  \mathbb{P}_{h,d}$, {the inherited finite element space on $\Sigma$}. 
Our full FEM-BEM algorithm is:
\begin{subequations}
\begin{itemize}
\item  {\bf Step 1:} Solve the finite dimensional system\label{eq:NhBEMFEM}
\begin{equation}
 \label{eq:NhBEMFEM:0} \left({\cal I}-\begin{bmatrix}
                                       &{\rm Q}^h_{\Sigma} \KGammaSigma^N\\
                                       {\rm Q}_N\KSigmaGamma^h
                                      \end{bmatrix}
\right)\begin{bmatrix}
                          f^h_{\Sigma}\\
                          f^N_\Gamma
                          \end{bmatrix}=\begin{bmatrix*}[r]
                          {\rm Q}^h_{\Sigma}\gamma_{\Sigma} u^{\rm inc}\\
                          - {\rm Q}_N \gamma_{\Gamma} u^{\rm inc} 
                          \end{bmatrix*}.
\end{equation} 

\item {\bf Step 2:} Construct the FEM-BEM solution 
\begin{equation}
 \label{eq:NhBEMFEM:01}
  u_h:=\KSigmaOmegaTwo^h f_\Sigma^h,\quad 
  \omega_N:=\KGammaOmegaOneC^Nf^N_\Gamma,\quad u_{h,N}:=\begin{cases}
     u_h,\quad&\text{in $\Omega_2$},\\
     \omega_N+u^{\rm inc},\quad&\text{in $\Omega_1^{\rm c}$}.
    \end{cases}
 \end{equation}
\end{itemize}
\end{subequations}

\begin{remark}
We have committed a slight  abuse of notation in the right-hand-side of  \eqref{eq:NhBEMFEM:0} by writing
\[
 {\rm Q}_N \gamma_{\Gamma} u^{\rm inc}
\]
instead of the correct, but more complex,  $
  {\rm Q}_N\big((\gamma_{\Gamma} u^{\rm inc})\circ{\bf x}\big)$. Similarly,  
\[
 {\rm Q}_N\big((\KSigmaGamma^h\,\cdot\,)\circ{\bf x}\big)
\]
should be read in the lower extra-diagonal block of the matrix in \eqref{eq:NhBEMFEM:0}. 
Indeed, this is equivalent to replacing a space on $\Gamma$ with that obtained via the parameterization \eqref{eq:parameterization}. Since both spaces are isomorphic, being strict in the notation for 
description of these operators is not {absolutely} necessary. In particular, we  {avoid} complicated notation and {use} a compact way to describe the algorithm 
and associated  theoretical results. 
\end{remark}

\begin{remark}

 Complete numerical analysis of the FEM-BEM algorithm is beyond the
scope of this article. 
%
%
In a future {work}, we shall carry out a detailed numerical analysis of the FEM-BEM algorithm. {Below we give the main results.} In summary, the analysis is based
on the following assumption on the mesh-grid:

\paragraph{\bf{Assumption 1}} There exists $\varepsilon_0>0$ such that  the sequence of grids $\{{\cal T}_h\}_h$ satisfies  
\begin{equation}\label{eq:mesh_res}
 h^{1/2}  h_D^{-\varepsilon_0}\to 0
\end{equation}
where $D\subset \Omega_2\setminus\overline{\Omega}_0$ is an open neighborhood of  $\Gamma$,  and $h_D$, the maximum of the diameters of the elements of the grid ${\cal T}_h$ with non-empty intersection with $D$.  

We note that this assumption allows locally refined grids, but introduces a very weak restriction on the ratio between the largest element in $\Omega_2$ and the smallest element in $D$. However, since the exact solution is smooth on $D$, {the partial differential equation in this domain is just the homogeneous Helmholtz equation}, and it is reasonable to expect that small elements are not  going to be used in this subdomain. 

\revB{Using Assumption 1, in a future work we shall prove the well-posedness of the  
discrete system~\eqref{eq:NhBEMFEM} and  optimal order of convergence of the FEM-BEM solution.
In particular, after deriving convergence of the individual FEM and BEM approximations, we shall prove
the following convergence result:} For any region $\Omega_R\subset\Omega_1^{\rm c}=\mathbb{R}^2\setminus \overline{\Omega}_1$, 
$0 < \epsilon \leq \epsilon_0$, 
$r\ge 0$, $t\ge d+3/2$,
\begin{eqnarray}
& & \hspace{-0.5in}  \|u-u_h\|_{H^{1}(\Omega_2)}+ \| \omega-\omega_N\|_{H^{r}(\Omega_R)}  \nonumber \\
&\le& C\big(h_D^{d-\varepsilon} N^{-\varepsilon}+h_\Sigma^{d+1/2}+N^{-t}+h_D^d\big)\|u^{\rm inc}\|_{H^{t+1}(\Omega_2)} 
 +C\inf_{v_h\in\mathbb{P}_{h,d}}\|u-v_h\|_{H^1(\Omega_2)}, \label{eq:fin_res}
\end{eqnarray}
where $h_D$ is as in~\eqref{eq:mesh_res} and $h_\Sigma$ is the maximum distance between any two consecutive Dirichlet/constrained nodes in $\mathcal{T}_h$;
$(u, \omega) = (\KSigmaOmegaTwo f_\Sigma,\KGammaOmegaOneC f_\Gamma)$ is the exact solution of \eqref{eq:BEMFEM}; and  $(u_h, \omega_N)$  is the unique solution of the numerical method \eqref{eq:NhBEMFEM}. 
\end{remark}

Next we describe algebraic details required for
implementation of the algorithm, followed by numerical experiments in Section~4 to demonstrate the efficiency of the FEM-BEM algorithm to simulate
wave propagation in the heterogeneous and unbounded medium. 

\subsection{FEM-BEM algebraic systems and evaluation of wave fields}

Simulation of   approximate interior and exterior wave fields $u_{h,N}$ using the solution of~\eqref{eq:NhBEMFEM:0} 
and the representation in~\eqref{eq:NhBEMFEM:01} requires: (i) computing the interior
solution $u_h$ by once solving  the finite element system~\eqref{eq:4.1} using the Dirichlet data~$f_{\Sigma}^h$;
and (ii) the exterior solution ${\omega_N}$ in $\Omega_1^{\rm c}$ by evaluating the layer potential value $(\mathrm{DL}_k^N-\mathrm{i}\mathrm{SL}_k^N){\cal L}_{k}^{N} f^N_\Gamma$, using the representation in~\eqref{eq:SLNDLN}.  

\revA{Since ${\cal L}_{k}^{N} f^N_\Gamma \in \mathbb{T}_N$ and that the dimension of $\mathbb{T}_N$ is $2N$,
using~\eqref{eq:def:Tn}{--\eqref{eq:SLNDLN}}, the degrees of freedom (DoF)  required to compute the exterior solution ${\omega_N}$  is equal to   the
number of interpolatory uniform grid points $t_j,~j=-N+1, \dots, N$ in~\eqref{eq:grid-points} that determine the interpolatory operator
$\mathrm{Q}_N$ in~\eqref{eq:interp}}. 
The linear algebraic system corresponding to the Dirichlet  
problem~\eqref{eq:4.1}  for  
$u_h \in  \mathbb{P}_{h,d}$ is obtained by using an ansatz that is  a linear combination of the basis functions spanning
$\mathbb{P}_{h,d}$. Coefficients in the $u_h$ ansatz are values of  $u_h$ at the nodes that determine $\{\mathcal{T}_h\}_h$.
The nodes include constrained/boundary Dirichlet nodes on $\Sigma$ and free/interior non-Dirichlet nodes in $\Omega_2$.

Henceforth, for a chosen mesh for the bounded domain $\Omega_2$, 
we use the notation  $M$ and $L$ to denote  the number of Dirichlet- and  free-nodes nodes in the mesh, respectively.
The FEM system~\eqref{eq:4.1} to compute  the solution $u_h$ 
leads to an $L$-dimensional linear system  for the unknown vector 
${\bf u}_L$ (that are values  of ${u}_h$ at the  interior nodes). The system is  governed  
by a real symmetric sparse matrix, say, ${\bf A}_{L}$. 
The matrix ${\bf A}_{L}$ is obtained by eliminating  the row and column vectors associated at the boundary nodes. 
Let  ${\bf D}_{L,M}$ be the  $L\times M$    matrix that is used to move the Dirichlet condition 
to the right-hand-side of the system. Thus for a given  Dirichlet data vector $\widehat{{\bf  f}}_{M}$, we may theoretically write 
${\bf u}_L = {\bf A}^{-1}_{L}{\bf D}_{L,M}  \widehat{{\bf  f}}_{M}$. Let  ${\bf T}_{2N,L}$ be  the $2N\times L$ {sparse} matrix so that
${\bf T}_{2N,L} {\bf u}_L (= {\bf T}_{2N,L}  {\bf A}^{-1}_{L}{\bf D}_{L,M}  \widehat{{\bf  f}}_{M})$  is the {trace} of  the finite element solution ${w}_h$ of~\eqref{eq:4.1} 
at the $2N$ interior points  $\mathbf{x}(t_j) \in \Gamma, j = -N+1, \dots, N$ that are the BEM grid points. 

For describing
the full FEM-BEM system, using the above representation, it is convenient to define the $2N \times M$ matrix 
\begin{equation}\label{eq:2N-M}
{\bf \widetilde{K}}_{2N,M}: = {\bf T}_{2N,L}  {\bf A}^{-1}_{L}{\bf D}_{L,M}.
\end{equation}
The matrix ${\bf A}^{-1}_{L}$  in~\eqref{eq:2N-M}, in general, should not be computed in practice. 
We may consider instead a ${\bf L}_L {\bf D}_{L}{\bf L}_L^\top$
 factorization~\cite{Duff:2004} (for example, implemented in the Matlab command {\tt ldl}), where ${\bf D}_L$ is a block diagonal matrix with $1\times 1$ or $2\times 2$ blocks and ${\bf L}_L$ is a block (compatible) unit  lower triangular matrix. Hence, each multiplication by $ {\bf A}^{-1}_{L}$ is reduced to solving  two (block) triangular and one $2\times 2$ block diagonal system which can be efficiently done, leading to evaluation of 
${\bf \widetilde{K}}_{2N,M}$ on $M$-dimensional vectors.
Of course the ${\bf L}_L {\bf D}_{L}{\bf L}_L^\top$ factorization is a relatively expensive process,   but worthwhile in our method to simulate the complex
heterogeneous and unbounded region model. (We further quantify this process using numerical experiments in Section~\ref{sec:num-exp}.)

The ansatz for  the unknown density $f^N_\Gamma \in \mathbb{T}_N$  is a linear combination of $2N$ known basis
functions ${\exp({\rm i}\ell t)}, ~\ell = -N+1, \dots, N$ in~\eqref{eq:def:Tn} that span $\mathbb{T}_N$.
The $2N$-dimensional BEM system  for the unknown vector 
$\widetilde{\bf f}_{2N}$ (that are values  of the unknown density at the  Nystr\"om  node points $t_j,~j=-N+1, \dots, N$)
is governed by a complex dense matrix and an input $2N$-dimensional  vector  $\widetilde{{\bf  f}}_{2N}$ determined
by the Dirichlet data on $\Gamma$ in the exterior homogeneous model~\eqref{eq:BEM:0} evaluated
at $t_j,~j=-N+1, \dots, N$. We may   write 
\begin{equation}\label{eq:bem_vec}
{\bf B}_{2N}{\bm{\varphi}}_{2N} = \widetilde{{\bf  f}}_{2N},
\end{equation}
 where
${\bf B}_{2N}$ is the $2N \times 2N$ Nystr\"om matrix corresponding to the discrete
boundary integral operator in~\eqref{eq:00}. Similar to  ${\bf T}_{2N,L}$, let ${\bf P}_{M,2N}$ be  the matrix  representation of the (discrete) combined potential generated by a density  at the  $M$ Dirichlet nodes of ${\cal T}_h$. That is,
${\bf P}_{M,2N} \bm{\varphi}_{2N}$ is the vector form of 
${\rm Q}_{h}^\Sigma \gamma_\Sigma(\mathrm{DL}_{k}^{N}-{\rm i}k\mathrm{SL}_{k}^{N})\varphi$, following the BEM representation~\eqref{eq:00} for evaluation
of the exterior field  at the $M$ Dirichlet nodes on $\Sigma$. Similar to the interior problem based matrix in~\eqref{eq:2N-M}, corresponding
to the exterior field it is convenient to introduce the $M \times 2N$ matrix 
\begin{equation}\label{eq:M-2N}
{\bf \widehat{K}}_{M,2N}: =  {\bf P}_{M,2N}{\bf B}_{2N}^{-1}.
\end{equation}
\revA{Obviously, $M<<L$ (since $M\sim L^{1/2}$ in the 2D case for quasi-uniform grids) and, thanks to the choice  of  the smooth boundary $\Gamma$, the standard Nystr\"om BEM is spectrally accurate, which further implies that $2N < < M$.   (We will quantify this  substantially smaller ``$<<$'' claim using numerical experiments in Section~\ref{sec:num-exp}.)}
Thus the cost of setting up an ${\bf LU}$ decomposition of the dense matrix ${\bf B}_{2N}$ is
negligible and consequently the matrix $\widehat{\bf K}_{M,2N}$ product with any $2N$-dimensional 
vector can be efficiently evaluated.  

The implementation procedure described above to compute the interior and exterior fields using~\eqref{eq:NhBEMFEM:01} requires 
 {the   $M$-dimensional vector $\widehat{{\bf  f}}_M$ with the values of the unknown  at the Dirichlet nodes}
  on $\Sigma$ and the $2N$-dimensional 
${\widetilde{{\bf  f}}_{2N}}$ at the $2N$ uniform grid points  ${\bf x}(t_j),~j= -N+1, \dots, N$ on $\Gamma$.
Since $\Sigma$ and $\Gamma$ are artificial  boundaries 
for the decomposition of the original model,  the vectors $\widehat{{\bf  f}}_M, \widetilde{{\bf  f}}_{2N}$ 
are unknown.  The interface system~\eqref{eq:NhBEMFEM:0}, that uses the data $u^{\rm inc}$ in the original model, 
completes the process to compute $\widehat{{\bf  f}}_M, \widetilde{{\bf  f}}_{2N}$. In particular, for 
the matrix-vector
form description of~\eqref{eq:NhBEMFEM:0},  we obtain  input data vectors, say  
$\widehat{{\bf u}}_{M}^{\rm inc}$ and $\widetilde{{\bf u}}_{2N}^{\rm inc}$, using the vector form representations of
${\rm Q}^h_{\Sigma}\gamma_{\Sigma} u^{\rm inc}$ and ${\rm Q}_N \gamma_{\Gamma} u^{\rm inc}$, respectively.

More precisely, using~\eqref{eq:2N-M}--\eqref{eq:M-2N}, the
matrix-vector algebraic system corresponding to ~\eqref{eq:NhBEMFEM:0}
takes the form
\begin{equation}\label{eq:system:01}
 \begin{bmatrix}
  {\bf I}_M & -{\bf \widehat{K}}_{M,2N}\\ \\
  -{\bf \widetilde{K}}_{2N,M}& {\bf I}_{2N} 
 \end{bmatrix}
 \begin{bmatrix}
 \widehat{{\bf  f}}_M\\ \\
 \widetilde{{\bf  f}}_{2N}
 \end{bmatrix}
=
 \begin{bmatrix}
  \widehat{{\bf u}}_{M}^{\rm inc}\\ \\ 
  -\widetilde{{\bf u}}_{2N}^{\rm inc}
 \end{bmatrix}
\end{equation} 
{where  ${\bf I}_{M}, {\bf I}_{2N}$ are, respectively, the $M \times M$ and $2N \times 2N$   identity matrices.}


In our implementation, instead of solving the full linear system in~\eqref{eq:system:01}  we work with the Schur complement 
\begin{subequations}\label{eq:system:02}
 \begin{eqnarray}
 (\underbrace{{\bf I}_{2N}-{\bf \widetilde{K}}_{2N,M}{\bf \widehat{K}}_{M,2N}}_{{=:}{\bf A}_{\rm Sch}})\widetilde{\bf f}_{2N} &=& 
 -\widetilde{\bf u}_{2N}^{\rm inc}+{\bf \widetilde{K}}_{2N,M}\widehat{\bf u}_M^{\rm inc}, \label{eq:system:02a}\\
 \widehat{\bf f}_M &=& \widehat{\bf u}_M^{\rm inc}+{\bf \widehat{K}}_{M,2N}\widetilde{\bf f}_{2N} \label{eq:system:02b}.
\end{eqnarray}
\end{subequations}
After solving for  $\widetilde{\bf f}_{2N}$ in~\eqref{eq:system:02a}, the main computational cost for finding 
$\widehat{\bf f}_M$ involves only the matrix-vector  multiplication  ${\bf \widehat{K}}_{M,2N}\widetilde{\bf f}_{2N}$. The latter
requires solving a BEM system, which can be carried out using a direct solve because  $2N$ is relatively small.

\section{Numerical experiments}\label{sec:num-exp}

In this section we consider two sets of numerical experiments to demonstrate the overlapping decomposition framework based
FEM-BEM algorithm. In the first set of experiments,   the heterogeneous domain $\Omega_0$ has non-trivial curved boundaries and the  refractive index function  
$n$ is smooth; and in the second set of experiments $\Omega_0$ is a complex non-smooth structure and \revB{$n$ is a discontinuous function}. For these two sets of  experiments, we consider the $\mathbb{P}_d$ Lagrange finite elements  with $d=2,3,4$
for the interior FEM model with mesh values $h$, and  several values of the Nystr\"om method parameter  $N$ to achieve spectral accuracy and to make the BEM errors
less than those in the FEM discretizations. The reported CPU times in the section are based on serial 
a implementation of the algorithm in Matlab (2017b)  on a desktop with a 10-core Xeon E5-2630  processor and $128$GB RAM.

\revA{In our numerical experiments to compute $\widetilde{\bf f}_{2N}$ in~\eqref{eq:system:02a}, we solve the linear
system using: (i)  the iterative GMRES method with the (relative) residual set to  $10^{-8}$ in all the cases; and (ii) the direct Gaussian elimination solve which requires the full matrix  ${\bf A}_{\rm Sch}$ in~\eqref{eq:system:02a}. Both  approaches are compared for the numerical experiments in Section~\ref{subsec:dir_iter_comp}. 
As an error indicator of our full FEM-BEM algorithm,
we analyze the widely used quantity of interest (QoI) in numerous wave propagation applications: the far-field  arising from both the interior and exterior fields induced by the incident field impinging from a particular direction. 
For a large class of inverse wave models~\cite{KressColton}, the {far-field}
measured at several directions 
is fundamental to infer various 
properties of the wave propagation medium.}

\revA{To  computationally verify the quality of our FEM-BEM algorithm in Section~\ref{sec:num-exp}, we analyze the numerical far-field error
	at thousands of  direction unit vectors $\bf z$. Using~\eqref{eq:bem_vec}, we define a   spectrally accurate approximation to the QoI as 
	\begin{equation}\label{eq:far-z}
	\left(\bm{\mathcal{F}}_N\bm{\varphi}_{2N}\right)({\bf z}):=  \sqrt{\frac{k}{8\pi} }\exp\big(-\tfrac14\pi{\rm i}  \big)\frac{\pi}{N}  \sum_{j=-N+1}^{N} \exp(
	-\mathrm{i} k({\bf z}\cdot{\bf x}(t_j))) \big[{\bf z}\cdot(x_2'(t_j),-x_1'(t_j))+1\big] \left[\bm{\varphi}_{2N}\right]_j. 
	\end{equation}
	The exact representation of the QoI is~\cite{KressColton} 
	\begin{equation}\label{eq:far}
	\left({\cal F}\varphi\right)({\bf z}):=
	\sqrt{\frac{k}{8\pi } } \exp\big(-\tfrac14\pi{\rm i}  \big)\int_{0}^{2\pi} \exp(
	-\mathrm{i} k ({\bf z}\cdot{\bf x}(t))) \big[{\bf z}\cdot(x_2'(t),-x_1'(t))+1\big]\varphi (t)\,{\rm d}t. 
	\end{equation}
	Using the angular representation of the direction vectors ${\bf z}$, 
	we compute approximate far-fields  at $1,000$ uniformly distributed angles. We report the QoI
	errors  for various  grid parameter sets $(h,N)$, and demonstrate high-order convergence of our FEM-BEM algorithm. 
	The maximum  of  the estimated errors in the approximate QoI,  using the values at the $1,000$ uniform directions, are used 
	below to validate the efficiency and high-order accuracy of the FEM-BEM algorithm. }
\subsection{Star-shaped domain with five-star-pointed refractive index} 
In Experiment 1 set, we choose $\Omega_0$ to be the star-shaped region
sketched in the interior of the disk $\Omega_1$ in  Figure~\ref{fig:expPatrickStar}, and the 
refractive index function is defined using  polar coordinates as 
\[
n^2(r,\theta):= 1+16 \chi\Big(\frac{1}{0.975}\Big[\frac{r}{2+0.75\sin(5\theta)}-0.025\Big]\Big),
\]
with 
\[
 \chi(x) := \frac{1}{2}(\widetilde{\chi}(x)+1-\widetilde{\chi}(1-x)),\quad
 \widetilde{\chi}(x) := \begin{cases}
 1, &\text{if $x\le 0$},\\
                        \exp\big(\frac{1}{e-e^{1/x}}\big),\quad&\text{if } x\in (0,1),\\
   0, &\text{if $x>1$}.            
                       \end{cases}
\]
Notice that $\widetilde{\chi}(x)$ is a smooth  cut-off  function with $\mathop{\rm supp}\chi =(-\infty,1]$. Therefore, the function $\chi$ is  smooth and also {\em symmetric} around $1/2$: $\chi(1-x)=1-\chi(x)$ for any $x$. 
    
\begin{figure}[ht]

\vspace{-0.1in}
\centerline{
 \includegraphics[width = 0.8\textwidth]{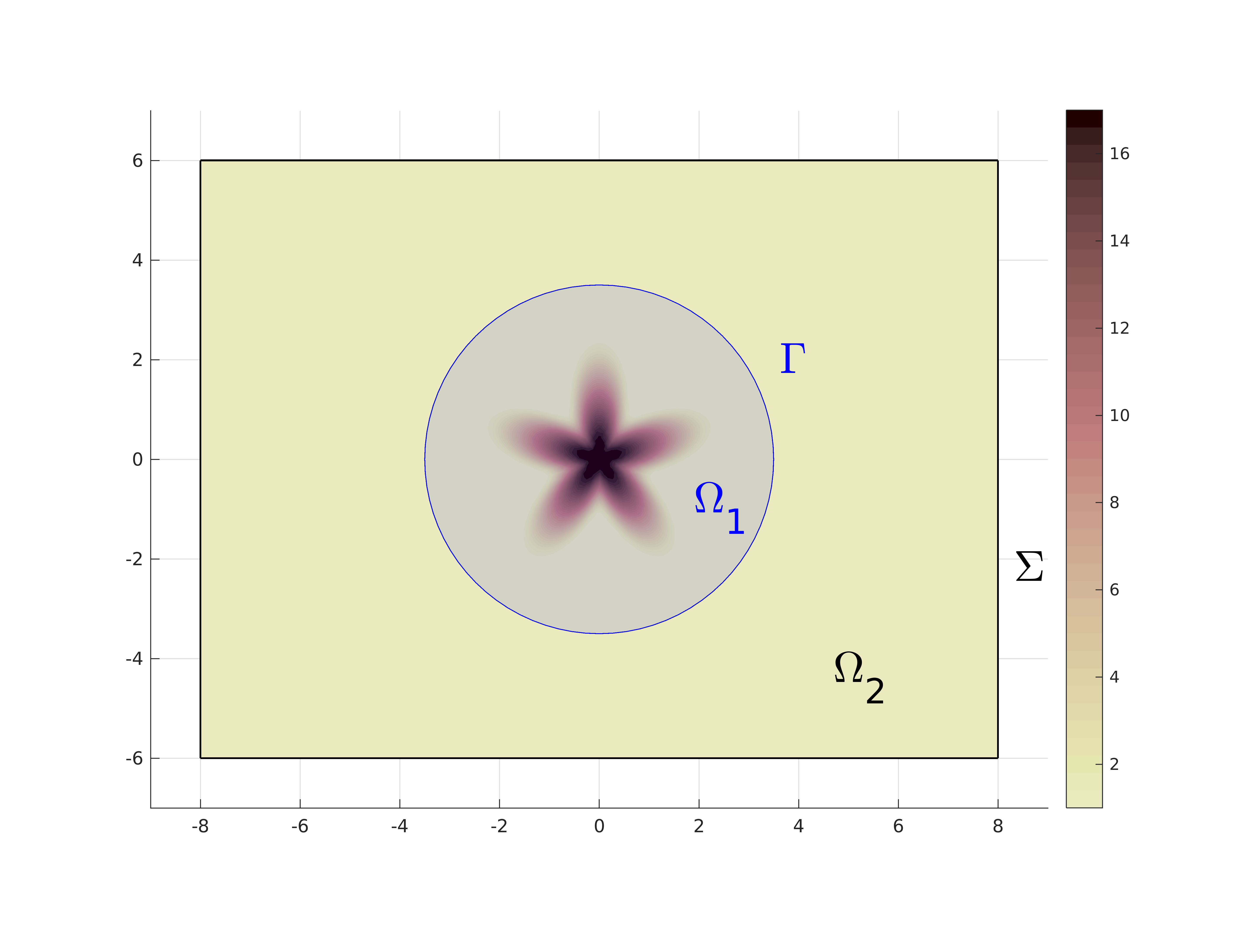}} 
 
 \vspace{-0.6in}
 \caption{\label{fig:expPatrickStar} Heterogeneous medium  and artificial boundaries for Experiment 1.}
\end{figure}


\revA{For this example, 
 $\Omega_2$  is the rectangle $[-6,6]\times [-8,8]$ with  boundary $\Sigma$, so that the
diameter of  the interior domain is $20$. Thus, for a chosen wavenumber $k$, the interior heterogeneous model is of wavelength $10k/\pi$. For our numerical experiments we choose three wavenumbers $k=\pi/4, \pi,4\pi$, to simulate the problems with acoustic characteristic size of $2.5, 10, 40$  wavelengths, respectively. The smooth boundary $\Gamma$ for this example is a  circle centered at zero and radius $3.5$}


For the interior FEM model,  the initial coarse grid consists of $2,654$ triangles, which is refined up to four times, in the usual way. We show the simulated  far-field error results in \revB{Tables~\ref{table:exp01:08} and \ref{table:exp01:09}
using $\mathbb{P}_3$ and  $\mathbb{P}_4$ elements, respectively.}  
In these tables  estimates of the 
(relative) maximum  errors  in computing the QoI far-fields  are presented as well as the number (given within parentheses) of GMRES iterations needed to achieve convergence with the residual tolerance  $10^{-8}$. Next we discuss some key aspects of the computed results  in~Tables~\ref{table:exp01:08}-\ref{table:exp01:09}.

To compute the errors for a set of discretization parameters, as  {\em exact/truth} solutions we used the FEM-BEM algorithm solutions obtained  with $N=640$ and the next level of FEM mesh refinement to these in the tables. The fast spectrally accurate  convergence of the  Nystr\"om BEM, after achieving a couple of
digits of accuracy, can be observed  by  following the far-field maximum errors in the last columns in Tables~\ref{table:exp01:08}-\ref{table:exp01:09}. 
In particular the last columns results, for the FEM  spline degree {$d =  3, 4$} cases, demonstrate that
relatively  small DoF $2N$ is required for the Nystr\"om BEM solutions  accuracy  to match that of the FEM solutions,
especially  compared to the FEM DoF $L$. 
The last rows in 
Tables~\ref{table:exp01:08}-\ref{table:exp01:09} clearly demonstrate that higher values of $N$  are 
not useful because of  the stagnation of the errors due to limited accuracy of the  FEM discretizations.
Further, a closer analysis of the results in Tables~\ref{table:exp01:08}-\ref{table:exp01:09}  shows that the 
the computed far-fields exhibit superconvergence, with ${\cal O}(h^{2d})$ errors. In addition, 
in Figure~\ref{fig:kpi} we demonstrate the  faster convergence of the  (Experiment 1) smooth total field solutions in the $H^1$- norm, 
and compare with the rate of convergence for  a non-smooth solution (Experiment 2) case.

In the Experiment 1 set, with a smooth heterogeneous region  $\Omega_0$ and a smooth refractive index 
function $n$, it can be shown that the exact near-field solution for the model problem is smooth.
However, this fact alone is not sufficient to explain in detail the superconvergence of the  computed far-fields.
We may conjecture that some faster convergence is {occurring} in the background for
the near-field in some weak norms, and that the calculation of the far-fields is benefitting  from this to 
achieve the superconvergence. In a future work, we shall explore the  numerical analysis  our FEM-BEM algorithm.

\begin{table}[h] \small\setlength{\tabcolsep}{4pt}

\begin{center}\tt
 
 \begin{tabular}{r|rl|rl|rl|rl|rl}
    $N/{L}$&\multicolumn{2}{c|}{7,999 }&   \multicolumn{2}{c|}{31,657}    &     \multicolumn{2}{c|}{125,953}   &     \multicolumn{2}{c|}{502,465}   &   \multicolumn{2}{c}{2,007,169}\\ 
 \hline&&&&&&&&\\      
    010  &   3.1e-03& (012) & 6.6e-05& (012)  &     2.2e-06& (012)  &  1.2e-06& (012)   & 1.2e-06& (012) \\
    020  &   3.1e-03& (012) & 6.5e-05& (012)  &     2.0e-06& (012)  &  2.5e-10& (012)   & 4.7e-11& (012) \\
    040  &   3.1e-03& (012) & 6.5e-05& (012)  &     2.0e-06& (012)  &  1.8e-10& (012)   & 1.4e-11& (012)  \\
    080  &   3.1e-03& (012) & 6.4e-05& (012)  &     2.0e-06& (012)  &  1.5e-10& (012)   & 9.0e-12& (012)  \\
  \end{tabular}
 \end{center}

\begin{center}\tt
 \begin{tabular}{r|rl|rl|rl|rl|rl}
    $N/{L}$&\multicolumn{2}{c|}{7,999 }&   \multicolumn{2}{c|}{31,657}    &     \multicolumn{2}{c|}{125,953}   &     \multicolumn{2}{c|}{502,465}   &   \multicolumn{2}{c}{2,007,169}\\ 
 \hline&&&&&&&&\\     
    010  &     4.3e-01 & (020) &  1.8e-01 & (020)  &  1.8e-01 & (020)   &  1.8e-01 & (020)   &   1.8e-01 & (020) \\
    020  &     3.5e-01 & (031) &  1.6e-02 & (031)  &  3.3e-04 & (031)   &  7.3e-06 & (031)   &   5.2e-06 & (031) \\
    040  &     3.5e-01 & (031) &  1.6e-02 & (031)  &  3.2e-04 & (031)   &  6.0e-06 & (031)   &   3.5e-07 & (031) \\
    080  &     3.5e-01 & (031) &  1.6e-02 & (031)  &  3.3e-04 & (031)   &  6.0e-06 & (031)   &   1.4e-07 & (031) \\
  \end{tabular}
 \end{center} 
  
   \begin{center}\tt
 \begin{tabular}{r|rl|rl|rl|rl|rl}
    $N/{L}$&\multicolumn{2}{c|}{7,999 }&   \multicolumn{2}{c|}{31,657}    &     \multicolumn{2}{c|}{125,953}   &     \multicolumn{2}{c|}{502,465}   &   \multicolumn{2}{c}{2,007,169}\\ 
 \hline&&&&&&&&\\        
    020  &    2.8e+00&(040) &  1.4e+00&(040)   & 1.1e+00 &(040)   &  1.4e+01&(040)  &   4.0e+00 &(040)\\  
    040  &    1.8e+00&(060) &  5.2e-01&(080)   & 6.0e-01 &(080)   &  9.1e-02&(080)  &   8.6e-02 &(080)\\  
    080  &    2.3e+00&(063) &  5.9e+00&(100)  & 6.3e-01 &(100)  &  4.7e-02&(102) &   8.3e-04 &(102)\\  
    160  &    2.2e+00&(063) &  5.0e+00&(100)  & 6.3e-01 &(100)  &  4.7e-02&(102) &   8.3e-04 &(102)
  \end{tabular}
  \end{center}
  \vspace{-0.2in}
  \caption{\label{table:exp01:08}Experiment 1: $\mathbb{P}_3$ Finite element space and 
  $k = \pi/4, \pi, 4\pi$ (top, middle, bottom tables). \revAB{In the first row and the first column, ${L}$ and $2N$ are the number of degrees of freedom used to compute  the FEM  and BEM  solutions, respectively. The number of GMRES iterations required for solving the system, with  a residual tolerance of $10^{-8}$, is given within the parenthesis.  Estimated (relative) uniform errors in the far-field are given in columns two to five.}}

  \end{table}

 
\begin{table}[h] \small
\begin{center}\tt

 \begin{tabular}{r|rl|rl|rl|rl|rl}
    $N/{L}$&   \multicolumn{2}{c|}{14,145}&   \multicolumn{2}{c|}{56,129} & \multicolumn{2}{c|}{223,617 } & \multicolumn{2}{c|}{892,673}  &  \multicolumn{2}{c}{3,567,105}\\ 
 \hline&&&&&&&&\\      
    010  &   3.9e-04 & (012) &  9.4e-06  & (012)  &     1.4e-06 & (012)  &  1.2e-06 & (012)   & 1.2e-06 & (012)\\
    020  &   3.9e-04 & (012) &  8.9e-06  & (012)  &     2.5e-07 & (012)  &  6.9e-10 & (012)   & 8.4e-11 & (012)\\
    040  &   3.9e-04 & (012) &  8.9e-06  & (012)  &     2.5e-07 & (012)  &  7.0e-10 & (012)   & 1.0e-10 & (012) \\
    080  &   3.9e-04 & (012) &  8.9e-06  & (012)  &     2.5e-07 & (012)  &  7.0e-10 & (012)   & 9.9e-11 & (012) \\
  \end{tabular}
 \end{center}

\begin{center}\tt
 \begin{tabular}{r|rl|rl|rl|rl|rl}
     $N/{L}$&   \multicolumn{2}{c|}{14,145}&   \multicolumn{2}{c|}{56,129} & \multicolumn{2}{c|}{223,617 } & \multicolumn{2}{c|}{892,673}  &  \multicolumn{2}{c}{3,567,105}\\ 
 \hline&&&&&&&&\\          
    010  &    2.0e-01 & (020)  &   1.8e-01 & (020) &   1.8e-01 & (020)    &   1.8e-01 & (020)  &1.8e-01 & (020) \\
    020  &    6.9e-02 & (031)  &   7.1e-04 & (031) &   6.9e-06 & (031)    &   5.4e-06 & (031)  &5.4e-06 & (031) \\
    040  &    6.9e-02 & (031)  &   7.1e-04 & (031) &   3.9e-06 & (031)    &   3.2e-08 & (031)  &4.7e-10 & (031) \\
    080  &    6.9e-02 & (031)  &   7.1e-04 & (031) &   4.0e-06 & (031)    &   2.4e-08 & (031)  &4.0e-10 & (031) \\
  \end{tabular}
  \end{center}

  \begin{center}\tt
 \begin{tabular}{r|rl|rl|rl|rl|rl}
    $N/{L}$&   \multicolumn{2}{c|}{14,145}&   \multicolumn{2}{c|}{56,129} & \multicolumn{2}{c|}{223,617 } & \multicolumn{2}{c|}{892,673}  &  \multicolumn{2}{c}{3,567,105}\\ 
 \hline&&&&&&&&\\          
    020  &     5.0e+00 & (040)  &   9.3e+00 & (040)  &   3.1e+00 & (040)       &   4.1e+00 & (040)   & 3.9e+00  & (040)\\
    040  &     3.7e+00 & (080)  &   4.9e-01 & (080)  &   2.4e-01 & (080)       &   8.5e-02 & (080)   & 8.6e-02  & (080)\\
    080  &     9.1e+00 & (098)  &   4.6e-01 & (100) &   2.6e-01 & (102)      &   2.0e-03 & (102)  & 8.8e-06  & (102) \\
    160  &     9.8e+00 & (098)  &   4.6e-01 & (100) &   2.6e-01 & (102)      &   2.0e-03 & (102)  & 8.8e-06  & (102)  
  \end{tabular}
  \end{center}
   \vspace{-0.2in}
  \caption{Experiment 1: \label{table:exp01:09} $\mathbb{P}_4$ Finite element space and 
$k = \pi/4, \pi, 4\pi$ (top, middle, bottom tables). \revAB{In the first row and the first column, ${L}$ and $2N$ are the number of degrees of freedom used to compute  the FEM  and BEM  solutions, respectively. The number of GMRES iterations required for solving the system, with a residual tolerance of $10^{-8}$, is given within the parenthesis.  Estimated (relative) uniform errors in the far-field are given in columns two to five.}}
 
\end{table}

\clearpage
  
 In Figure~\ref{fig:gmres}, we illustrate the convergence of the GMRES iterations and show that as
the frequency is increased four-fold, 
\revA{the number of required iterations for the solutions to converge with the $10^{-8}$ residual  tolerance
increases at (a slightly) slower rate.}
 \begin{figure}[h]
\centerline{
 \includegraphics[width = .8\textwidth ]{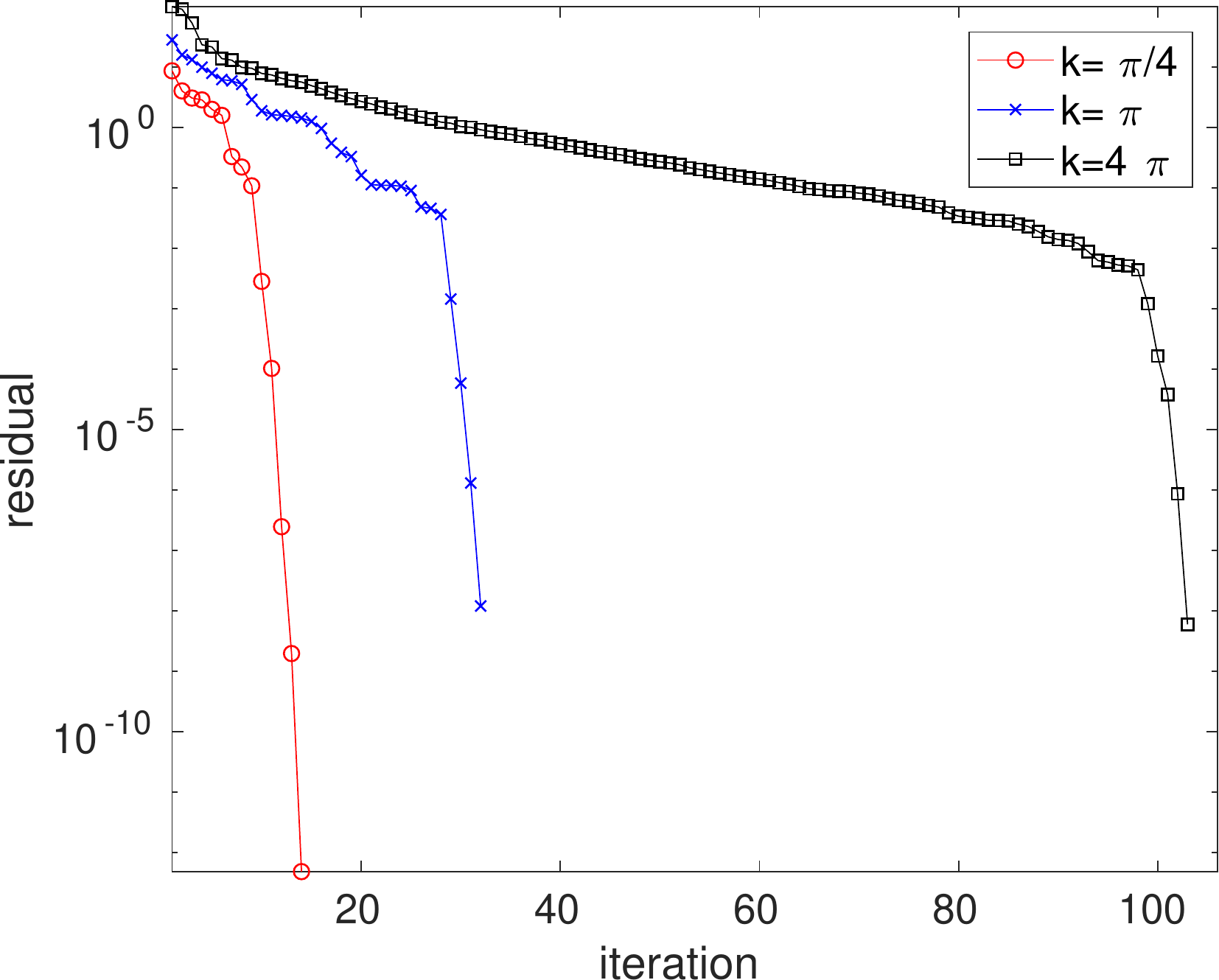}} 
 \caption{\label{fig:gmres} \revA{Number of GMRES iterations and  residual errors for   Experiment 1 simulations with $k=\pi/4$, $k=\pi$ and $k=4\pi$ 
 using the  $\mathbb{P}_3$ finite element space on  a grid 
 with the FEM DoF $L= 502,465$  and the BEM DoF $2N=160$.}}
\end{figure}
 
\revA{Next we consider   how the size of the overlapped FEM-BEM region $\Omega_{12}$ affects the speed of convergence of the GMRES iterations. To this end,  we have run a set of additional experiments for the star-shaped (Experiment 1) problem with $k = \pi$,  using several
choices of $\Gamma$, to obtain  larger to smaller diameter overlapped regions $\Omega_{12}$. In particular, we chose several BEM smooth boundaries
$\Gamma$  to be  circles centered at the origin with radii spanning from 2.625 (closer to the heterogeneity) to 5.856 (closer to  the 
FEM boundary $\Sigma$), yielding several  $\Omega_{12}$, respectively, with larger to smaller sizes. For all these  simulation cases, we fixed the BEM DoF to be  $2N = 160$, and the fixed $\mathbb{P}_3$ elements were obtained  using $445,440$ triangles with the number of free-nodes (FEM DoF) to be 
$L = 1,106,385$.  We present the corresponding results in Figure~\ref{fig:new}.}

\revA{In the left panel of Figure~\ref{fig:new}, we can see a sample of the curves $\Gamma$ used for the set of experiments with varying size $\Omega_{12}$,
and correspondingly in the right panel of Figure \ref{fig:new}, we present the number of GMRES iterations required to converge with, again, the residual tolerance $10^{-8}$. Results in Figure \ref{fig:new} clearly demonstrate that  the number of GMRES iterations increases as the size of the overlapped region $\Omega_{12}$ decreases. This can be explained as follows: At the continuous level, the interacting operators $\KGammaSigma$ and 
$\KSigmaGamma$ tend to lose the compactness property, as the overlapped region becomes thinner. (We shall explore this 
observation theoretically in a future work.). 
On the other hand, it is interesting to note  from these experiments that the choice of $\Gamma$ being very close to the heterogeneity  does not affect the 
convergence of the GMRES iterations. We could conjecture that this might happen for the considered set of experiments because the exact solution 
for Experiment 1 problem is smooth. However, we have noticed a similar  behavior for the next   Experiment 2 problem, with a complex 
non-smooth heterogeneous region, for which regularity of the total wave field is limited.}
\begin{figure}
	\[
\includegraphics[width=0.51\textwidth]{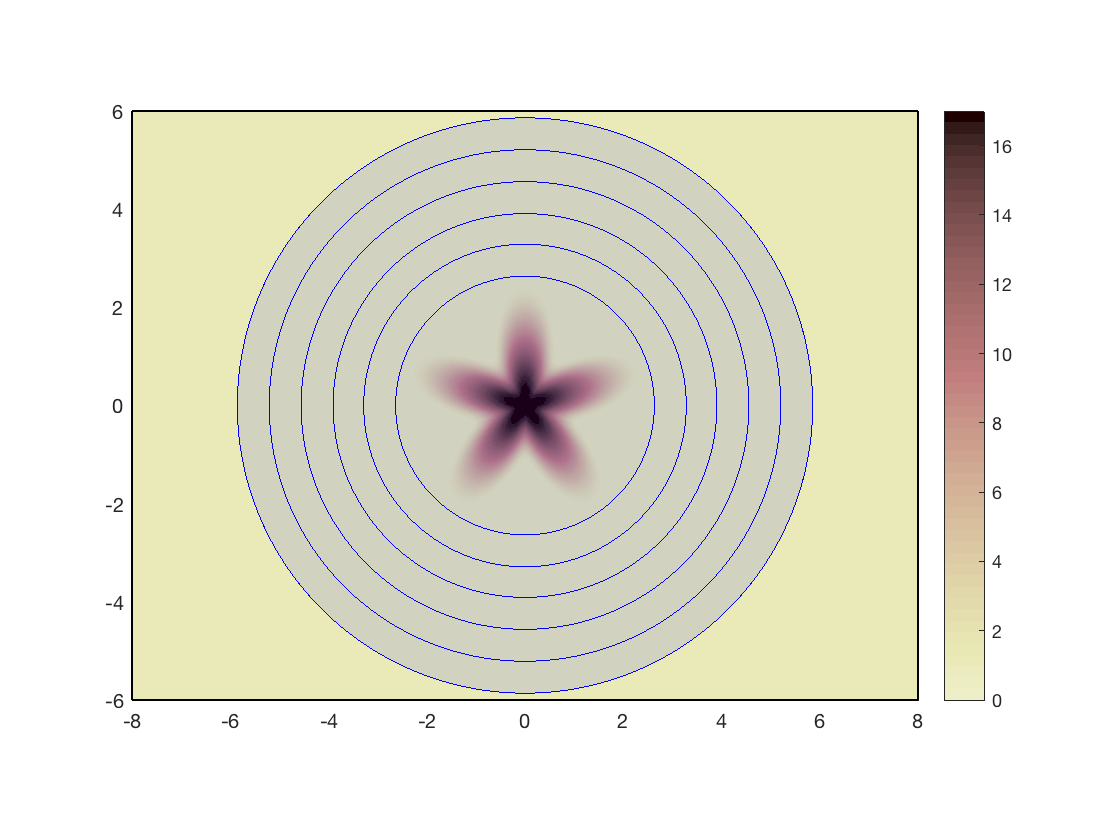}
\includegraphics[width = 0.52\textwidth]{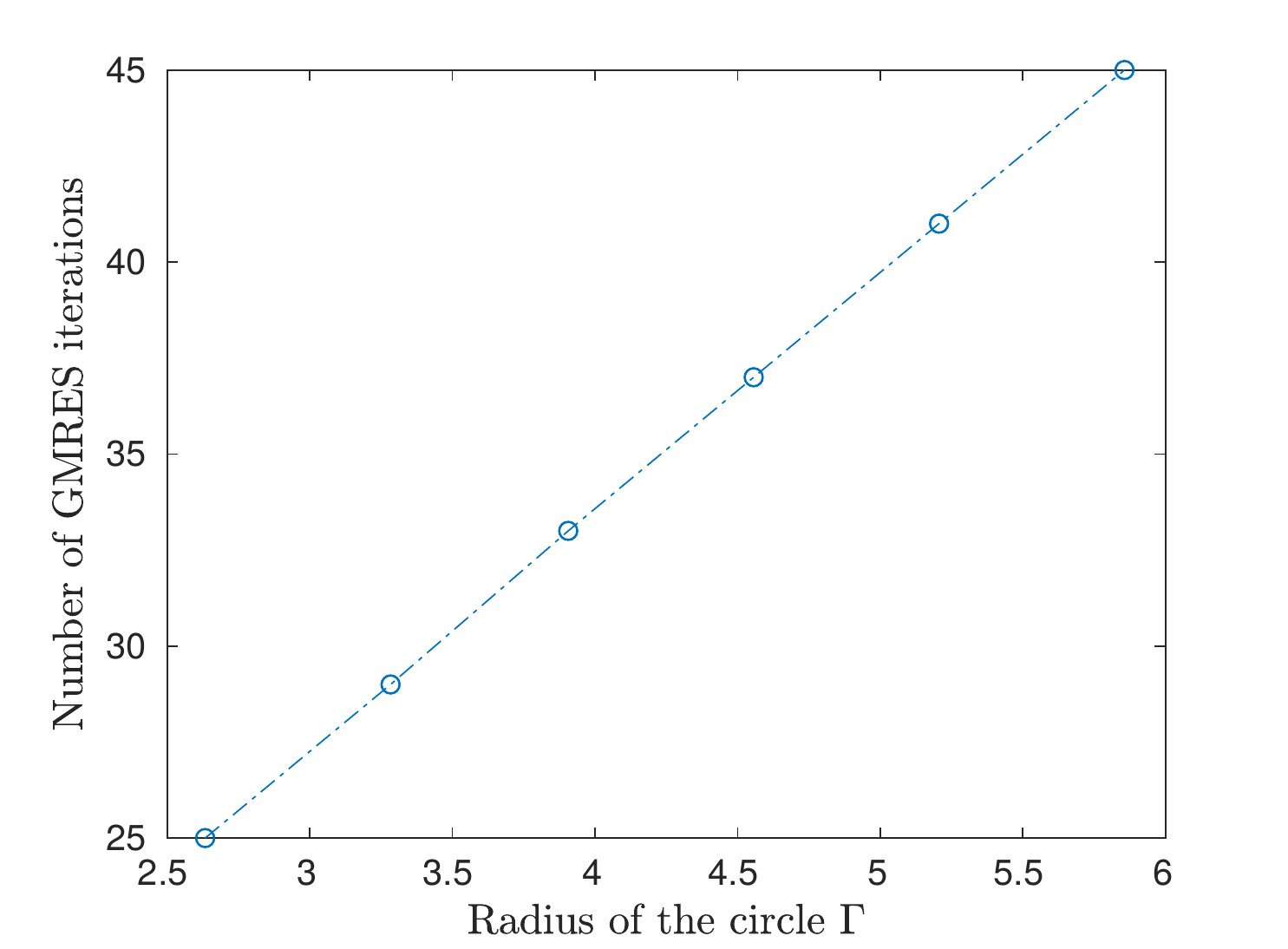}
\]
\vspace{-0.5in}
\caption{\label{fig:new}\revA{Dependence of the number of GMRES iterations on the size of the overlapping region: On the left, various choices of 
the smooth (circular) interface $\Gamma$. On the right, 
radii  of the circles $\Gamma$ vs. number of GMRES iterations required for convergence with a  residual tolerance of $10^{-8}$.}}
\end{figure}

\subsection{{\em Pikachu}-shaped domain with piecewise smooth refractive index}

%

\begin{figure}[h]
\centerline{
  \includegraphics[height = .6\textwidth]{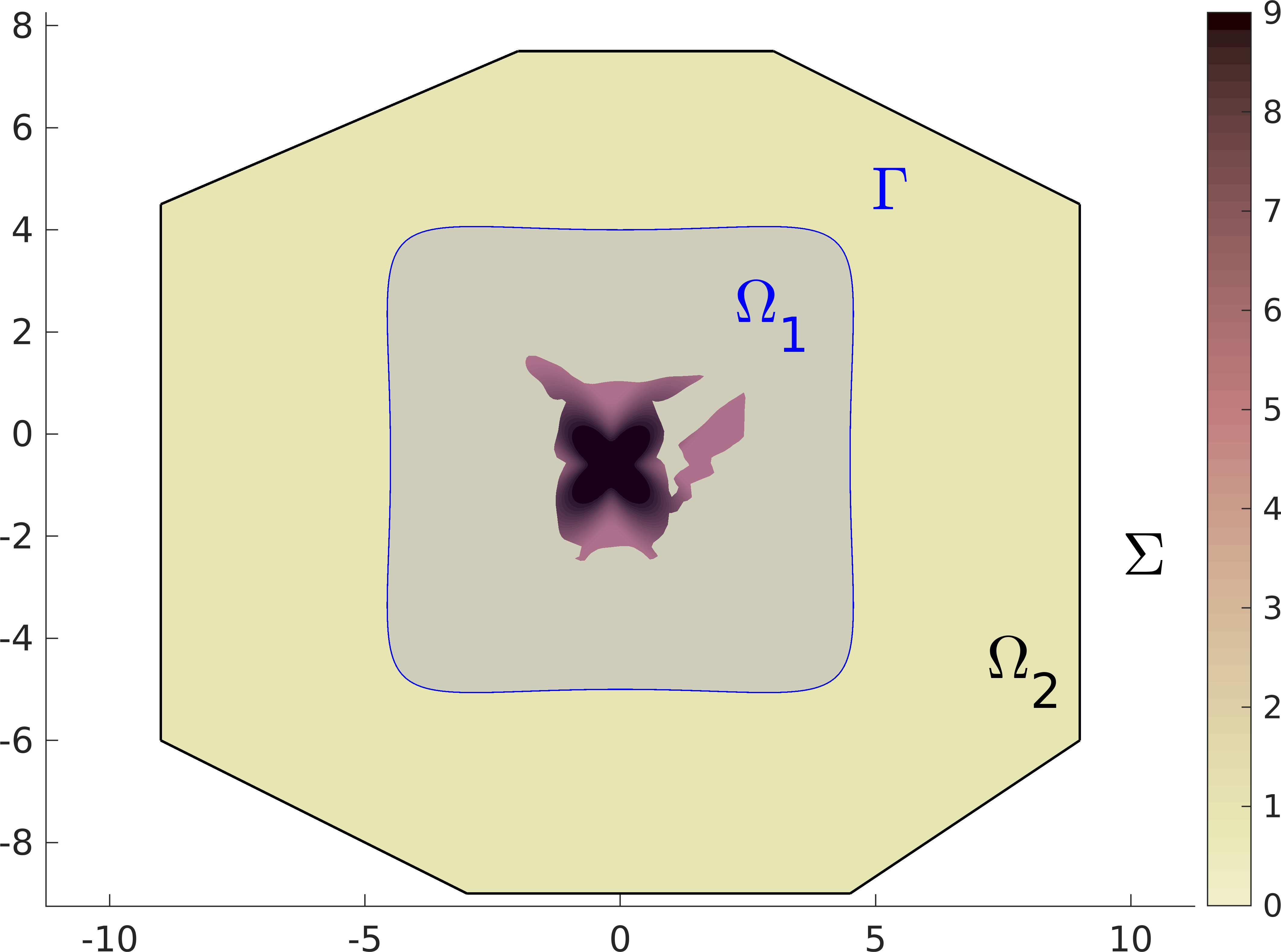}} 
 \caption{\label{fig:expPikachu} Pikachu  heterogeneous domain and artificial boundaries $\Gamma$  and $\Sigma$  for Experiment 2.}
\end{figure}
In  Experiment 2 set of experiments,  we  consider a more complicated non-smooth heterogeneous region  shown in the interior
of  the curved domain $\Omega_1$ in Figure~\ref{fig:expPikachu}. The region $\Omega_0$ is set to be a  polygonal {\em Pikachu}-shaped domain 
with the discontinuous refractive index function 
\[
n^2(x,y) := \begin{cases}
       5 + 4\chi\Big(\frac{1}{0.9}\Big[\frac{r}{2-0.75\cos(4\theta)}-0.025\Big]\Big), &(x,y)\in \Omega_0,\\
        1 , &(x,y)\not\in \Omega_0,\\
         \end{cases}
\]
where $r = \sqrt{(x + 0.18)^2 + (y + 0.6)^2}, \theta = \arctan \! 2((y + 0.6),(x + 0.18))$.
The grids used in our computation, are adapted to the region $\Omega_0$, in such a way that  any triangle $\tau \in{\cal T}_h$ is either contained  or has empty intersection with $\Omega_0$. As the boundary of $\Omega_1$ and  for the smooth curve $\Gamma$ for the exterior model,  we choose  
\[
 {\bf x}(t) = \frac{7\sqrt{2}  }4\Big(   (1+\cos^2 t)\cos t + (1+\sin^2 t)\sin t,(1+\sin^2 t)\sin t-  (1+\cos^2 t)\cos t \Big)
\]
For the interior FEM model, we choose $\Omega_2$ to be a polygonal domain as in Figure~\ref{fig:expPikachu} with boundary $\Sigma$. 
We then proceed as in the previous experiment, using  an initial coarse grid with $8,634$ triangles which is refined up to four times.
The solution $u$ of the model is not smooth {in $\Omega_0$ and $\overline{\Omega}_0^{c}$,
because of the non-smoothness} of the region $\Omega_0$ and {the jump in the refractive index function}. 
One may consider the use of a graded mesh around the boundary of $\Omega_0$ to obtain faster convergence. Based on 
the size of $\Omega_2$, 
the choices  \revA{$k = \pi/4, \pi, 4\pi$} lead to approximately $2.5$, $10$, and $40$ wavelengths interior FEM model, respectively, for simulations in Experiment 2. 

We observe from the  integer numbers (within in parentheses) in Tables~\ref{table:exp02:03}-\ref{table:exp02:05} that  the number of GMRES iterations grow, slower than the  quadruple growth of the three frequencies considered in Experiment 2.  The estimated (relative) maximum far-field errors for the non-smooth  Experiment 2 model are given in Tables~\ref{table:exp02:03}-\ref{table:exp02:05}, demonstrating high-order accuracy of our FEM-BEM model as the finite element space degree, grid size, and the BEM DoF are increased.
 In Figure~\ref{fig:kpi}, for $d =2, 3, 4$, we compare   convergence of the total field in the $H^1$-norm for the  smooth (Experiment 1) and non-smooth (Experiment 2)  simulations.

In Figure~\ref{fig:exp02:01} we depict the simulated wave field  solution for $k =\pi$, with $\mathbb{P}_4$ finite elements on a grid with $138,144$ triangles and \revB{$L=1,106,385$} free-nodes for the FEM solution, and \revB{$2N=320$} for the BEM solution. Specifically, we plot  the simulated 
absorbed and scattered field numerical  solution $u_{h,N}$ inside $\Omega_2$  in  Figure~\ref{fig:exp02:01}.

\begin{table}[th] \small 
 \begin{center} \tt           \small
  \begin{tabular}{r|rl|rl|rl|rl}
    $N/{L}$&
                  \multicolumn{2}{c|}{ 39,085}  &    \multicolumn{2}{c|}{ 69,381}     &      \multicolumn{2}{c|}{ 622,573 }   &       \multicolumn{2}{c}{ 2,488,441}   \\ 
 \hline &&&&&&&\\         
    010  &   2.8e-03 &(015)  & 2.8e-03  &(015)   & 2.8e-03 &(015)     & 2.8e-03   &(015) \\
    020  &   5.8e-05 &(015)  & 8.4e-07  &(015)   & 8.4e-07 &(015)     & 8.4e-07   &(015)\\
    040  &   5.3e-05 &(015)  & 1.0e-07  &(015)   & 6.1e-09 &(015)     & 6.9e-10   &(015)\\
    080  &   5.8e-05 &(015)  & 7.4e-08  &(015)   & 6.5e-09 &(015)     & 4.4e-10   &(015)\\
  \end{tabular} 
  
  \end{center}
 \begin{center} \tt           \small
  \begin{tabular}{r|rl|rl|rl|rl}
   $ N/{L}$&
                  \multicolumn{2}{c|}{ 39,085}  &    \multicolumn{2}{c|}{ 69,381}     &      \multicolumn{2}{c|}{ 622,573 }   &       \multicolumn{2}{c}{ 2,488,441} \\   
 \hline &&&&&&&\\        
    020  &   2.5e+00 &(040)  & 2.5e+00 &(040)   & 2.5e+00 &(040)     & 2.5e+00 &(040)   \\
    040  &   3.8e-03 &(042)  & 2.5e-04 &(042)   & 7.1e-05 &(042)     & 5.2e-05 &(042)   \\
    080  &   3.1e-03 &(042)  & 1.7e-04 &(042)   & 7.1e-06 &(042)     & 2.7e-07 &(042)   \\
    160  &   3.4e-03 &(042)  & 1.4e-04 &(042)   & 7.9e-06 &(042)     & 2.6e-07 &(042)   \\
  \end{tabular} 
  \end{center}
 \begin{center} \tt           \small
  \begin{tabular}{r|rl|rl|rl|rl}
   $ N/{L}$&
                \multicolumn{2}{c|}{ 39,085}  &    \multicolumn{2}{c|}{ 69,381}     &      \multicolumn{2}{c|}{ 622,573 }   &       \multicolumn{2}{c}{ 2,488,441} \\   
 \hline&&&&&& \\       
    040  &    6.8e+00  &(080)  &3.2e+00  &(080)  &   3.7e+00 &(080)   &  3.6e+00 &(080)  \\
    080  &    9.2e+00  &(130) &7.4e-01  &(140) &   2.1e-00 &(139)  &  2.3e+00 &(139)  \\
    160  &    6.7e+00  &(140) &4.6e-01  &(148) &   1.3e-02 &(149)  &  4.1e-04 &(149)  \\
    320  &    6.8e+00  &(140) &4.4e-01  &(148) &   1.1e-02 &(149)  &  2.8e-04 &(149)  
  \end{tabular} 
   \end{center}
\caption{\label{table:exp02:03}Experiment 2:  $\mathbb{P}_3$ Finite element space 
and  $k = \pi/4, \pi, 4\pi$ (top, middle, bottom tables). \revAB{In the first row and the first column, ${L}$ and $2N$ are the number of degrees of freedom used to compute  the FEM  and BEM  solutions, respectively. The number of GMRES iterations required for solving the system, with  residual tolerance of $10^{-8}$, is given within the parenthesis.  Estimated (relative) uniform errors in the far-field are given in columns two to five.}}
\end{table}

\begin{table}[t] \small \small                                                                                    

\begin{center} \tt           \small
   \begin{tabular}{r|rl|rl|rl|rl}         
    $N/{L}$&                  \multicolumn{2}{c|}{ 69,381}  &    \multicolumn{2}{c|}{276,905}     &    \multicolumn{2}{c|}{ 1,106,385}  & \multicolumn{2}{c}{4,423,073}    \\         
 \hline&&&&&& \\       
    010  &    2.8e-03  &(015)      &2.8e-03  &(015)  &   2.8e-03 &(015)   &  2.8e-03 &(015)  \\
    020  &    1.3e-06  &(015)  &8.4e-07  &(015)  &   8.4e-07 &(015)   &  8.4e-07 &(015)  \\
    040  &    1.3e-06  &(015)  &1.6e-07  &(015)  &   6.8e-10 &(015)   &  6.8e-10 &(015)  \\
    080  &    1.3e-06  &(015)  &1.6e-07  &(015)  &   6.9e-10 &(015)   &  6.8e-10 &(015)  \\
  \end{tabular}  
\end{center}  
 \begin{center} \tt           \small                                                                   
  \begin{tabular}{r|rl|rl|rl|rl}                                                                    
    $N/{L}$&                                                                                      
                  \multicolumn{2}{c|}{ 69,381}  &    \multicolumn{2}{c|}{276,905}     & \multicolumn{2}{c|}{1,106,385}  & \multicolumn{2}{c}{4,423,073 }   \\                                                                      
 \hline &&&&&&&\\                                                                                                                                           
    020    &  2.5e+00 &(040)   &  2.5e+00 &(040)  &   2.5e+00 &(040)     &   2.5e+00 &(040)    \\                                            
    040    &  2.8e-04 &(042)   &  4.7e-05 &(042)  &   5.3e-07 &(042)     &   5.2e-05 &(042)    \\                                          
    080    &  1.9e-04 &(042)   &  2.2e-06 &(042)  &   1.8e-07 &(042)     &   6.9e-09 &(042)    \\                                          
    160    &  1.6e-04 &(042)   &  1.1e-06 &(042)  &   6.3e-08 &(042)     &   3.3e-09 &(042)    \\                                          
  \end{tabular}                                                                                       
  \end{center}                                                                                                     
 \begin{center} \tt           \small                                                                               
  \begin{tabular}{r|rl|rl|rl|rl}                                                                                
    $N/{L}$&                  \multicolumn{2}{c|}{ 69,381}  &    \multicolumn{2}{c|}{276,905}     &    \multicolumn{2}{c|}{ 1,106,385}  & \multicolumn{2}{c}{4,423,073}    \\                    
 \hline &&&&&&&\\                                                                                                          
    040    &   1.7e+00 &(080)      &  3.8e+00 &(080)  &   3.6e+00 &(080)   &   3.6e+00 &(080) \\   
    080    &   8.8e-01 &(139)     &  1.8e+00 &(140) &   2.2e+00 &(139)  &   2.3e+00 &(139) \\   
    160    &   5.4e-01 &(147)     &  3.9e-02 &(149) &   4.8e-04 &(149)  &   6.9e-05 &(149) \\   
    320    &   5.4e-01 &(147)     &  3.6e-02 &(149) &   2.9e-04 &(149)  &   8.4e-06 &(149)     
  \end{tabular}                                                                                                    
  \end{center}                                                                    
\caption{\label{table:exp02:05}Experiment 2:  $\mathbb{P}_4$ Finite element space  
and $k = \pi/4, \pi, 4\pi$ (top, middle, bottom tables). \revAB{In the first row and the first column, ${L}$ and $2N$ are the number of degrees of freedom used to compute  the FEM  and BEM  solutions, respectively. The number of GMRES iterations required for solving the system, with  residual tolerance of $10^{-8}$, is given within the parenthesis.  Estimated (relative) uniform errors in the far-field are given in columns two to five.}}
\end{table} 

\clearpage


\begin{figure} [!ht]
\centerline{\includegraphics[width=0.66 \textwidth]{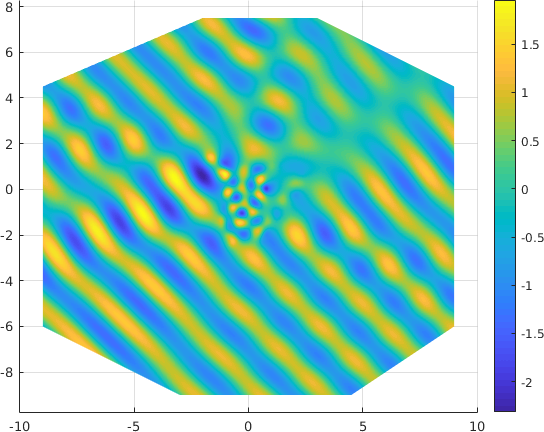}}
\caption{\label{fig:exp02:01}Real part  of the total field FEM  solution $u_h$  in $\Omega_2$ for $k=\pi$.}


 \centerline{\includegraphics[width=0.80\textwidth]{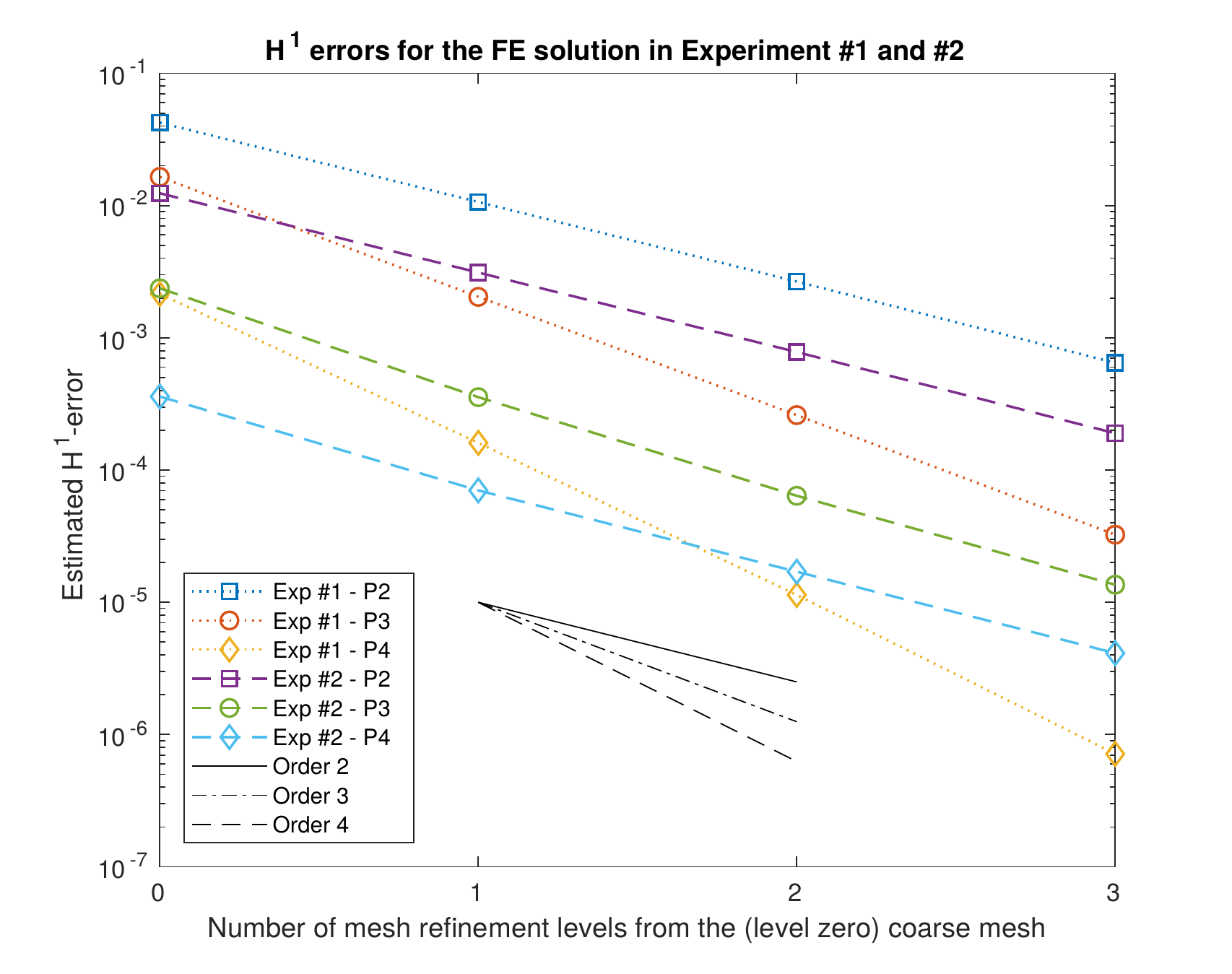}}
 \caption{\label{fig:kpi}Comparisons of convergence of the FEM-BEM algorithm for the total field in  the $H^1(\Omega_2)$-norm for Experiment 1 and 2
 using $\mathbb{P}_2$, $\mathbb{P}_3$ and $\mathbb{P}_4$ elements with $N = 80$ {and $k=\pi/4$}. The bottom part of the figure shows  the expected  order of convergence, as given in~\eqref{eq:fin_res}, for {\em smooth} solutions.}
\end{figure}

\clearpage 
\subsection{Direct solver implementation and comparison with  iterative solver}\label{subsec:dir_iter_comp}
    
In this subsection we  discuss the direct solver implementation of our method and compare its performance with the iterative approach we have used for simulating results described earlier in the section. When computing the matrix  in~\eqref{eq:system:02a}, the main issue is concerned with the matrix  ${\bf \widetilde{K}}_{2N,M}$, 
{which comprises the calculation of finite element solution followed by its evaluation at the nodes of the} {BEM.}
Because of the spectral accuracy of the Nystr\"om  BEM 
approximation, the   DoF $2N$ is expected to be smaller, in practice, even compared to the number $M$ of  FEM boundary  Dirichlet (constrained) nodes
(that is, $M > 2N$). Accordingly, in  our implementation we use instead the representation 
\[
 {\bf \widetilde{K}}_{2N,M}^\top = ( {\bf T}_{ 2N,L}{\bf A}^{-1}_{L}  {\bf D}_{L,M})^\top =
 {\bf D}^\top_{L,M}{\bf L}_L^{-1}{\bf D}_L^{-1}{\bf L}_L^{-\top}  {\bf T}_{ 2N,L}^\top, 
\]
where we recall that   ${\bf A}_{L} ={\bf L}_L {\bf D}_{L}{\bf L}_L^\top$ is symmetric. This representation requires solving $2N$ (independent) finite element problems, one for each {column of $\widetilde{\bf K}_{2N,M}^\top$,  and  a (sparse) matrix-vector multiplication. The first process,  consumes the bulk of computation time (but is a naturally parallel task w.r.t. $N$) and can be carried out with wall-clock time similar to solving one FEM problem~\cite[Section 5.1.5]{GaMor:2016}.

The common CPU time for the direct and iterative solver amounts to the assembly of the finite element matrices ${\bf A}_L$ and ${\bf D}_{L,M}$,  the ${\bf L}  {\bf D}  {\bf L}^\top$ factorization of the former,   the boundary element matrix ${\bf B}_{2N}$  and the auxiliary matrices   ${\bf T}_{2N,L}$ and ${\bf P}_{M,2N}$. Consequently the major difference  in computation between
the two approaches is: (i) the construction and  storage of the matrix in~\eqref{eq:system:02a}, followed by
exactly solving the linear system for the direct method; versus (ii) the setting up of the system~\eqref{eq:system:02a} for  matrix-vector multiplication and
approximately solving the linear system {with}  the GMRES iterations. The former approach is faster especially 
if the number of GMRES iterations is not very low
(in single-digits) because of modern fast multi-threaded implementation of the direct solver. However, the latter approach is memory efficient and needed
especially for large scale 3-D models. 

Using a desktop machine, with a $10$-core processor and $128$GB RAM, we were able to apply the direct solver to simulate the example 2-D models in Experiment 1 and 2, 
even with millions of  FEM (sparse) DoF within our FEM-BEM framework . For one of the largest cases reported
in Table~\ref{table:exp01:09}, {with} $\mathbb{P}_4$ elements for the wavenumber $k = 4\pi$ ($40$ wavelengths case), with 
\[
N = 80, \quad  L = 3,567,105 \quad \text{{(with $445, 440$ triangles), and}}  \quad M = 7,168
\]
the GMRES approach system setup CPU time was $172$ seconds; and the direct approach  setup CPU time was $332$ seconds. Because
of requiring $102$ GMRES iterations, the solve time to compute a converged iterative solution was {\bf 586 seconds}. However,
because of the very efficient multi-threaded direct solvers (in Matlab) the direct solve time to compute the exact solution was only {\bf 0.014 seconds}.

The size of the  interface linear system  for the experiment is only $160 \times 160$ and hence our algorithm   can be very efficiently used
for a large  number of incident waves $u^{{\rm inc}}$, that occur only in the small interface system. 
Thus we conclude that our FEM-BEM framework provides options to apply direct or iterative approaches to efficiently simulate wave propagation
in heterogeneous and unbounded media. For 2-D low and medium frequency models with sufficient RAM, it seems to be efficient even to use
the direct solver, and for higher frequency cases  iterative solvers are efficient because of the demonstrated well-conditioned property of
the system. 

\section{Conclusions}

\revAB{In this article we developed a novel continuous and  discrete computational  framework for an equivalent reformulation and efficient simulation of  
an absorbed and scattered wave propagation model, respectively, in a bounded heterogeneous medium and an unbounded homogeneous free-space. The 
model is governed by the Helmholtz equation and a decay radiation condition at infinity. The decomposed framework  incorporates
the radiation condition exactly and is based on creating two overlapping regions, without truncating the full space unbounded propagation medium. 
The overlapping  framework has the advantage of choosing a smooth artificial boundary for the unbounded region of the reformulation, and a simple polygonal/polyhedral boundary  for the bounded part of the two regions. The advantage facilitates the application of a spectrally accurate BEM for
approximating the scattered wave, and setting up a high-order FEM for simulating the absorbed wave. We prove the equivalence of the
decomposed overlapping continuous framework and the given model. The efficiency of our two-dimensional FEM-BEM computational framework was demonstrated in this work  using  two sets of numerical experiments, one  comprising a smooth  and the other a non-smooth heterogeneous medium.}

\section*{Acknowledgement}

V\'{\i}ctor Dom\'{\i}nguez thanks the support of the project MTM2017-83490-P. Francisco-Javier Sayas was partially supported by the NSF grant DMS-1818867.
\bibliography{biblio}

\bibliographystyle{plain}

\end{document}